\documentclass[12pt,verbatim]{amsart}
\usepackage{amsfonts}
\usepackage{amsthm}
\usepackage{amsmath}
\usepackage{amssymb}
\usepackage{amscd}
\usepackage{subfigure}
\usepackage{ifthen}
\usepackage{graphicx}
\usepackage{caption}

\usepackage{xcolor}
\usepackage[active]{srcltx}

\newcounter{minutes}
\divide\time by 60
\newcounter{hours}
\multiply\time by 60 \addtocounter{minutes}{-\time}

\newcommand{\comment}[1]{}

\newtheorem{theorem}[equation]{Theorem}
\newtheorem{lemma}[equation]{Lemma}
\newtheorem{corollary}[equation]{Corollary}
\newtheorem{proposition}[equation]{Proposition}
\theoremstyle{definition}
\newtheorem{definition}[equation]{Definition}
\newtheorem{example}[equation]{Example}

\newcommand{\beq}{\begin{equation}}
\newcommand{\eeq}{\end{equation}}
\newtheorem{remark}[equation]{Remark}
\numberwithin{equation}{section}

\newcommand{\RN}{\mathbb{R}^n}

\newcommand{\B}{{\mathbb{B}}}
\newcommand{\C}{{\mathbb{C}}}

\newcommand{\N}{{\mathbb{N}}}

\newcommand{\R}{{\mathbb{R}}}

\renewcommand{\Re}{{\rm Re}}       
\def\NABLA#1{{\mathop{\nabla\kern-.5ex\lower1ex\hbox{$#1$}}}}
\def\Nabla#1{\nabla\kern-.5ex{}_{#1}}
\def\Tabla#1{\Tilde\nabla\kern-.5ex{}_{#1}}
\renewcommand{\Tilde}{\widetilde}

\newtheorem{nonsec}[equation]{}

\makeatother

\begin{document}

\def\thefootnote{}
\footnotetext{ \texttt{\tiny File:~\jobname .tex,
          printed: \number\year-\number\month-\number\day,
          \thehours.\ifnum\theminutes<10{0}\fi\theminutes}
} \makeatletter\def\thefootnote{\@arabic\c@footnote}\makeatother

\title{On Quasi-inversions}

\author{David Kalaj}
\address{University of Montenegro, faculty of natural sciences and mathematics,
Cetinjski put b.b. 81000, Podgorica, Montenegro}
\email{davidk@t-com.me}

\author{Matti Vuorinen}
\address{Department of Mathematics and Statistics, University of Turku, 20014
Turku, Finland} \email{vuorinen@utu.fi}

\author{Gendi Wang}
\address{Department of Mathematics and Statistics, University of Turku, 20014
Turku, Finland} \email{genwan@utu.fi}

\maketitle

\begin{abstract}
Given a bounded domain $D \subset  {\mathbb R}^n$ strictly starlike with
respect to $0 \in D\,,$ we define a quasi-inversion w.r.t. the boundary
$\partial D \,.$ We show that the quasi-inversion is bi-Lipschitz w.r.t.
the chordal metric if and only if every "tangent line" of $\partial D$
is far away from the origin. Moreover,  the bi-Lipschitz constant tends
to $1,$ when $\partial D$ approaches the unit sphere in a suitable way.
For the formulation of our results we use the concept of
the $\alpha$-tangent condition due to F. W. Gehring and J. V\"ais\"al\"a
(Acta Math. 1965). This condition is shown to be equivalent to the bi-Lipschitz and quasiconformal extension property of what we call the polar parametrization of  $\partial D$.
In addition, we show that the polar parametrization, which is a mapping of the unit sphere onto $\partial D\,,$ is bi-Lipschitz if and only if $D$ satisfies the $\alpha$-tangent condition.
\end{abstract}

{\small \sc Keywords.} {Quasi-inversion, starlike domain, $\alpha$-tangent condition}

{\small \sc 2010 Mathematics Subject Classification.}{26A16(30C62, 51F99, 53A04)}
\comment{
 \makeatletter
\renewcommand*\subjclass[2][2000]{%
  \def\@subjclass{#2}%
  \@ifundefined{subjclassname@#1}{%
    \ClassWarning{\@classname}{Unknown edition (#1) of Mathematics
      Subject Classification; using '1991'.}%
  }{%
    \@xp\let\@xp\subjclassname\csname subjclassname@#1\endcsname
  }%
}
 \makeatother}

\section{Introduction}

M\"obius transformations in $\overline{\R}^n=\R\cup\{\infty\}$ can be defined as mappings of the form $F=h_1\circ\cdots\circ h_m$, where each $h_j,\,j=1,\cdots,m$, is either a reflection in a hyperplane of $\R^n$ or an inversion in a sphere $S^{n-1}(a,r)=\{x\in\RN: |x-a|=r\}$, a mapping of the form \cite{ah2, be}
\beq\label{rh}
x\mapsto h(x)= a+\frac{r^2(x-a)}{|x-a|^2},\,\,\,\,x\in\R^n\setminus\{a\},
\eeq
and $h(a)=\infty$, $h(\infty)=a$, where $a\in\R^n$ and $r>0$.
M\"obius transformations have an important role in the study of the hyperbolic geometry of the unit ball $\mathbb{B}^n$ or of the upper half space
 $\mathbb{H}^n=\{(x_1,\cdots,x_n)\in\RN: x_n>0\}$ because the hyperbolic metric remains invariant under M\"obius automorphisms of the corresponding space. This
 invariance property immediately follows from the characterizing property of M\"obius transformations \cite[p. 32, Theorem 3.2.7]{be}:
the absolute ratio of every quadruple stays invariant under a M\"obius transformation $f$ 
\beq
|f(a),f(b),f(c),f(d)|=|a,b,c,d|.
\eeq
One of the basic facts is the distance formula for the mapping $h$ in (\ref{rh}) \cite{ah2},\cite[3.1.5]{be}, \cite[1.5]{vu}:
\beq\label{hf}
|h(x)-h(y)|=\frac{r^2|x-y|}{|x-a||y-a|}.
\eeq

Because M\"obius transformations are conformal maps they also have a role in the theory of quasiconformal maps which is the motivation of the present study. Our starting point is to define a quasi-inversion with respect to the boundary $\mathcal{M}=\partial D$ of a bounded domain $D\subset\RN$, strictly starlike w.r.t. the origin.  Given a point $x\in\R^n\setminus\{0\}$, consider the ray $L_x=\{tx: t>0\}$ and write  $r_x=|w|$ where $w$ is the unique point in $L_x\cap\mathcal{M}$. Then $r_x=r_{sx}$ for all $s>0$. We define the quasi-inversion in $\mathcal{M}$ by
\beq
x\mapsto f_{\mathcal{M}}(x)=\frac{r^2_x}{|x|^2 }x,\,\,\,\,x\in\R^n\setminus\{0\},
\eeq
and $f_{\mathcal{M}}(0)=\infty$, $f_{\mathcal{M}}(\infty)=0$ (e.g. see Figure \ref{qi}). Then $f_{\mathcal{M}}(x)=x$ for $x\in {\mathcal{M}}$ and $f_{\mathcal{M}}(D)=\overline{\R}^n\setminus\overline{D}$.

\begin{figure}[ht]
\includegraphics[height=5.2cm]{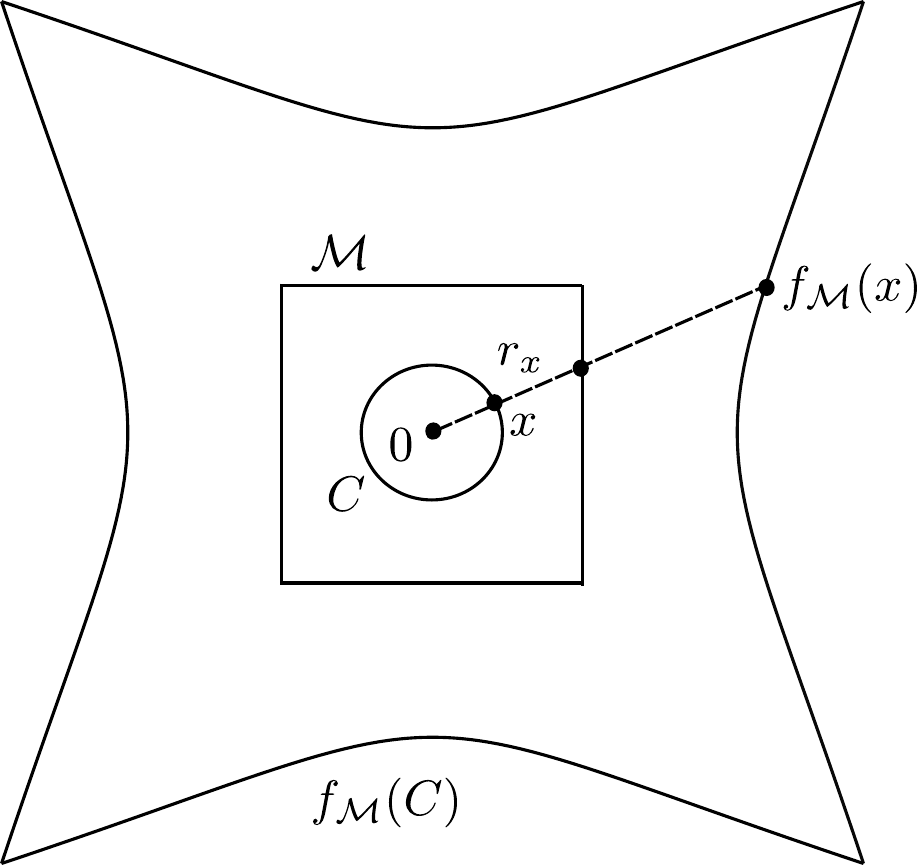}
\caption{\label{qi} The image $f_{\mathcal{M}}(C)$
of the circle $C=S^1(0,1/2)$ under the quasi-inversion $f_{\mathcal{M}}\,.$
Here $\mathcal{M}$ is the boundary of the square centered at $0$ with side length $2\,.$ }
\end{figure}
 Obviously the quasi-inversion in $S^{n-1}(0,r)$ coincides with the inversion $h$ in (\ref{rh}) when $a=0$. Therefore it is a natural question to study which properties of inversions hold for quasi-inversions, too.

Perhaps the simplest question is to investigate under which conditions on $\partial D$ we have a counterpart of the identify (\ref{hf}) in the form of an inequality, with a constant depending on $\partial D$. This is the content of our first main result, formulated as Theorem \ref{khh}. Secondly, it is a basic fact that the inversion in (\ref{rh}) with $a=0$ and $r=1$ is a 1-bi-Lipschitz (isometry) mapping w.r.t. the chordal metric $q$ (see Section 2). The result here for quasi-inversions is formulated as Theorem~\ref{qinv1}. Third, because the inversion $h$ in $\partial \mathbb{B}^n$ transforms the points $x,\,y\in\mathbb{B}^n$ to $h(x),\,h(y)\in\overline{\R}^n\setminus\overline{\mathbb{B}}^n$ with the equal hyperbolic distances $\rho_{\mathbb{B}^n}(x,y)=\rho_{\overline{\R}^n\setminus\overline{\mathbb{B}}^n}(h(x),h(y))$
we may look for a similar result in the quasi-inversion case. This question requires finding a suitable counterpart of the hyperbolic metric for a strictly starlike domain $D$ and its complementary domain $\overline{\R}^n\setminus\overline{D}$. Presumably several metrics could be used here. In our third main result we generalize this to the case of a M\"obius invariant metric introduced by P. Seittenranta, see Theorem \ref{qinv2}.
Moreover, several other results will be proved. It should be pointed out that the idea of quasi-inversion (without the name) was used by M. Fait, J. G. Krzyz, and J. Zygmunt in \cite{fkz}.
Some of the notions we use are close to the study of F. W. Gehring and
J. V\"ais\"al\"a \cite{gv} concerning quasiconformal mapping problems in $\R^3$. We refine some of these results and also a later result by
O. Martio and U. Srebro \cite{mar2}.

For a nonempty subset  $\mathcal{A}\subset\R^n\setminus\{0\}\,,$
the radial projection $\Pi: \mathcal{A}\rightarrow S^{n-1}$ is defined by
\beq\label{rpx}
\Pi(x)=x/|x|.
\eeq
Let $\mathcal{M}\subset\RN$ be
now the boundary of a domain in $\R^n$,  strictly starlike w.r.t. the origin.
Then the inverse function of $\Pi:\mathcal{M}\rightarrow S^{n-1}$ is denoted by
$\varphi: S^{n-1}\rightarrow\mathcal{M}$. Further, we define the radial
extension of $\varphi$ by
\beq\label{re}
\varphi_a: x\mapsto|x|^a\varphi(x/|x|),\,\,\,\,x\in\R^n\setminus\{0\},
\eeq
and $\varphi_a(0)=0$, $\varphi_a(\infty)=\infty$, where $a>0$. Note that for $\mathcal{M}= S^{n-1}$ this mapping is the standard radial stretching \cite[p. 49]{vais}.
The main result of
Section~3 is Theorem~\ref{multi}, which states roughly that the radial projection of $\mathcal{M}$ is bi-Lipschitz if and only if the domain bounded by $\mathcal{M}$ satisfies the $\alpha$-tangent condition, and the bi-Lipschitz constant depends only on $\alpha$ and the distance from origin to $\mathcal{M}$. Our results imply the unexpected fact that the global bi-Lipschitz constant of the radial projection of a hyper-surface onto the unit sphere is equal to the maximal value of the local bi-Lipschitz constant. The main result of Section~4 is Theorem~\ref{star1} which says that $\varphi_1$ is bi-Lipschitz if and only if the domain bounded by $\mathcal{M}$ satisfies the $\alpha$-tangent condition, with an explicit bi-Lipschitz constant. Moreover, $\varphi_a$ is quasiconformal if and only if the domain bounded by $\mathcal{M}$ satisfies the $\alpha$-tangent condition.
Further, we give an explicit constant of quasiconformality in terms of the angle $\alpha\in(0,\pi/2]$ and $a$.
We finish the paper by proving that the quasi-inversion in $\mathcal{M}$ is a $K$-quasiconformal mapping with $K=\cot^2\frac{\alpha}{2}$, see Theorem \ref{qinv3} in Section~5.

\section{Preliminaries}
Let $B^n(x,r)$ be the ball centered at $x$ with radius $r>0$ and $S^{n-1}(x,r)$
its boundary sphere. We abbreviate
$B^n(r)= B^n(0,r),\, S^{n-1}(r)= S^{n-1}(0,r),\, {\mathbb B}^n = B^n(1),\, S^{n-1}= S^{n-1}(1)$.
Let $[a,b]$ be the the segment with endpoints $a,\,b$. 

\begin{nonsec}{\bf Dilatations.}
Let $D,D'$ be subdomains of $\R^n$ and $f:D\to D'$ be a differentiable
homeomorphism and denote its Jacobian by $J(x,f)\,, x\in D.$
If $x\in D $ and $J(x,f)\neq 0\,,$ then the derivative of $f$ at
$x\in D$ is a bijective linear mapping $f'(x):\RN\rightarrow \RN$  and we denote
\begin{eqnarray}\label{lineardil}
H_I(f'(x))=\frac{|J(x,f)|}{\lambda_f(x)^n},\ \ \ H_O(f'(x))=\frac{\Lambda_f(x)^n}{|J(x,f)|},\ \ \ H(f'(x))=\frac{\Lambda_f(x)}{\lambda_f(x)},
\end{eqnarray}
where  $$\Lambda_f(x):=\max\{|f'(x)h|: |h|=1\}\,\,\, {\rm and } \,\, \, \, \lambda_f(x):=\min\{|f'(x)h|: |h|=1\}.$$
Sometimes instead of $\Lambda_f(x)$ we use notation $|f'(x)|$, to denote the norm of the matrix $A=f'(x)$. If $\lambda_1^2\le \dots\le \lambda_n^2$ ($\lambda_i>0, i=1,2,\cdots,n$) are eigenvalues of the symmetric matrix $AA^t$ where $A^t$ is the adjoint of $A$, then we have the following well-known formulas
\begin{equation}\label{formula}
|J(x,f)|=\prod_{k=1}^n\lambda_k,\ \  \ \Lambda_f(x)=\lambda_n,\ \ \  \lambda_f(x)=\lambda_1.
\end{equation}

By \eqref{lineardil} and \eqref{formula}, we arrive at the following simple inequalities \cite[14.3]{vais}
\begin{equation}\label{io1}
H(f'(x))\le \min\{H_I(f'(x)), H_O(f'(x))\}\le H(f'(x))^{n/2}.
\end{equation}
\begin{equation}\label{io2}
 H(f'(x))^{n/2}\le \max\{H_I(f'(x)),H_O(f'(x))\}\le H(f'(x))^{n-1}.
\end{equation}

The quantities
$$K_I(f)=\sup_{x\in D}H_I(f'(x)),\ \ \  K_O(f)=\sup_{x\in D}H_O(f'(x))$$
are called the inner and outer dilatation of $f$, respectively.
The maximal dilatation of $f$ is
$$K(f)=\max\{K_I(f),K_O(f)\}.$$
\end{nonsec}

\begin{nonsec}{\bf Quasiconformal mappings.}
In the literature, see e.g. \cite{car}, we can find various definitions of quasiconformality which are equivalent.
The following analytic definition for quasiconformal mappings is from \cite[Theorem 34.6]{vais}:
a homeomorphism $f: D\rightarrow D'$ is $C$-quasiconformal if and only if the following conditions are satisfied: (i) $f$ is ACL; (ii) $f$ is differentiable a.e.; (iii) $\Lambda_f(x)^n/C\le|J(x,f)|\le C\lambda_f(x)^n$ for a.e. $x\in D$. By \cite[Theorem 34.4]{vais}, if $f$ satisfies the conditions (i), (ii) and $J(x,f)\neq 0$ a.e., then
$$K_I(f)={\rm ess}\sup_{x\in D}H_I(f'(x)),\ \ \  K_O(f)={\rm ess}\sup_{x\in D}H_O(f'(x)).$$
Hence (iii) can be written as $K(f)\le C$ which by \eqref{io2} is equivalent to
\beq\label{k1}
H(f'(x))\le K \text{ for a.e. $x\in D$}.
\eeq
Here the constant $K\le C^{2/n}$.
In this paper we say that a quasiconformal mapping $f: D\to D'$ is $K$-quasiconformal if $K$ satisfies \eqref{k1}.

It is important to notice that  $f$ is $K$-quasiconformal if and only if $f^{-1}$ is $K$-quasiconformal and that the composition of $K_1$ and $K_2$ quasiconformal mappings is  $K_1K_2$-quasiconformal. (
It is well-known that this also holds for $K$-quasiconformality
in V\"ais\"al\"a's sense, see \cite[Corollary 13.3, Corollary 13.4]{vais}).

Recall that for the case of planar differentiable mapping $f$, we have
\beq\label{lamd1}
\Lambda_f(z)=|f_z|+|f_{\bar{z}}|=|f'(z)|,
\eeq
\beq\label{lamd2}
\lambda_f(z)=\left||f_z|-|f_{\bar{z}}|\right|=\frac{1}{|(f^{-1})'(z)|}.
\eeq
Hence the condition \eqref{k1} can be written as
$$
|\mu_f(z)|\le k \,\,{\rm a.e.\, on}\,\,D\,\,\,{\rm where}\, \,\,k=\frac{K-1}{K+1}
$$
and $\mu_f(z)= f_{\bar{z}}/f_z$ is the complex dilatation of $f$.
Sometimes instead of $K$-quasiconformal we write $k$-quasiconformal.
\end{nonsec}

\begin{nonsec}{\bf Lipschitz mappings.}
 Let $(X,d_X)$ and $(Y,d_Y)$ be metric spaces. Let $f: X\rightarrow Y$ be continuous. We say that $f$ is $L$-Lipschitz, if
\begin{eqnarray}\label{l1}
d_Y(f(x),f(y))\leq L d_X(x,y),\,\,{\rm for}\,\, x,\,y\in X,
\end{eqnarray}
and $L$-bi-Lipschitz, if $f$ is a homeomorphism and
\begin{eqnarray}\label{l2}
d_X(x,y)/L\leq d_Y(f(x),f(y))\leq L d_X(x,y),\,\,{\rm for}\,\,x,\,y\in X.
\end{eqnarray}
A $1$-bi-Lipschitz mapping is called an isometry. The smallest $L$ for which \eqref{l1} holds will be denoted by ${\rm Lip}(f)$.

Let $f: D\rightarrow fD$ be a Lipschitz map, where $D, fD\subset \R^n$
and $D$ is convex. By the Mean value theorem, we have the following simple fact
\beq\label{bil}
{\rm Lip}(f)=\mathrm{ess}\sup_{x\in D}|f'(x)|.
\eeq
Recall that Lipschitz maps are a.e. differentiable by the Rademacher-Stepanov theorem \cite[p. 97]{vais}.
\end{nonsec}

\begin{nonsec}{\bf Starlike domains.}
A bounded domain $D\subset\RN$ is said
to be {\it strictly starlike w.r.t. the point $a$} if each ray emanating
from $a$ meets $\partial D$ at exactly one point.
\end{nonsec}

\begin{nonsec}{\bf $\alpha$-tangent condition.} Suppose $D\subset\RN$ is a strictly
starlike domain w.r.t. the origin and $x\in\partial D$.
For each $z\in\partial D$, $z\neq x$, we let $\alpha(z,x)$
denote the acute angle which the segment $[z,x]$ makes with the ray from $0$ through $x$, and we define
$$\alpha(x)=\liminf_{z\rightarrow x}\alpha(z, x)\in [0,\pi/2].$$
If $\partial D$  has a tangent hyperplane at $x$ whose normal forms an acute angle $\theta$ with the ray from $0$ through $x$, then $$\alpha(x)=\pi/2-\theta.$$
We say a domain $D$ (and its boundary $\partial D$) satisfies the {\it $\alpha$-tangent condition} if for every $x\in\partial D$ we have $\alpha(x)\ge \alpha\in(0,\pi/2]$. If we do not specify $\alpha$ then we simply say  that $D$ satisfies the {\it tangent condition}.

F. W. Gehring and J. V\"ais\"al\"a \cite[Theorem 5.1]{gv} studied the dilatations of a homeomorphism in terms of the above $\alpha$-tangent condition (without the name).
\end{nonsec}

\begin{nonsec}{\bf $\beta$-cone condition {\rm ( Martio and Srebro \cite{mar2})}.}
Let $D\subset\RN$ be a bounded domain with $0\in D$ and let $\beta\in(0,\pi/4]$.
We say that $D$ satisfies the {\it $\beta$-cone condition} if the open cone
$$C(x,\beta):=\{z\in\R^n:|z-x|<|x|,
\left<x-z,x\right>>|x-z||x|\cos \beta\}$$
with vertex $x$ and central angle $\beta$ lies in $D$ whenever $x\in\partial D$. Note that if $D$ satisfies the $\beta$-cone condition, then $D$ is strictly starlike.
\end{nonsec}


\begin{proposition}\label{stat}
A domain $D$ satisfies the $\beta$-cone condition if and only if it satisfies the $\alpha$-tangent condition.
\end{proposition}
\begin{proof}
It is obvious that $D$ satisfies the $\alpha$-tangent condition with $\alpha=\beta$ if $D$ satisfies the $\beta$-cone condition.

We want to show that if $D$ satisfies the $\alpha$-tangent condition, then
$D$ satisfies the $\beta$-cone condition.
First, fix $x\in\partial D$, we prove that there exists a constant
$\beta_x\in(0,\pi/4]$, such that $C(x,\beta_x)\subset D$. If not, then
for all $ n\in\N$, there exists a sequence $ a_n\in C(x, 1/n)\cap \partial D$.
Since $\partial D$ is compact there is a subsequence still denoted by $\{a_n\}$
converging to a limiting point $a\in (0,x]\cap\partial D$. If $a=x$,
then $\alpha(x)\le\lim\inf_{n\rightarrow \infty}\alpha(a_n,x)=0$
which contradicts the $\alpha$-tangent condition.
If $a\in(0,x)$, a contradiction with  the starlikeness of $D$ follows.

Next, we prove that there exists a uniform constant $\beta\in (0,\pi/4]$ for all $x\in\partial D$, such that $C(x,\beta)\subset D$. Suppose not,
then for all $n\in\N$, there exists $x_n\in\partial D$ such that $C(x_n, 1/n)\cap \partial D\neq\emptyset$. So there exist a sequence of points $(y_n)$ from the boundary, with $y_n\neq x_n$ such that $\alpha(x_n,y_n)$ tends to zero when $n$ tends to infinity. By \cite[Eq.~5.11]{gv} we infer that the radial projection $\Pi$ of $\partial D$ onto the unit sphere $S$ is a bi-Lipschitz mapping (see section~3). Let $\varphi$ be its inverse, then there is a constant $L\ge 1$ such that \begin{equation}\label{lipip}
\frac{|\varphi(\eta)-\varphi(\zeta)|}{|\eta-\zeta|}\le L,\ \ \  \eta,\zeta\in S.
\end{equation}
 Then $\varphi(x)=r(x) x$ for some positive continuous function $r$ in $S$. Let $r_{\min}$ and $r_{\max}$ be the minimum and maximum of $r$ respectively. Then we have $x_n=\varphi(\eta_n)$, $y_n=\varphi(\zeta_n)$, for some $x_n,y_n\in S^{n-1}$ and we obtain
    $$\cos\alpha(x_n,y_n)=\frac{\left<\varphi(\eta_n)-\varphi(\zeta_n),\varphi(\eta_n)\right>}{|\varphi(\eta_n)-\varphi(\zeta_n)|\cdot|\varphi(\eta_n)|}.$$ Thus
\[   \begin{split}\cos^2\alpha(x_n,y_n)-1&
=\frac{\left<\varphi(\eta_n)-\varphi(\zeta_n),\varphi(\eta_n)\right>^2}{|\varphi(\eta_n)-\varphi(\zeta_n)|^2\cdot|\varphi(\eta_n)|^2}-1
\\&=\frac{(r^2(\eta_n)-r(\eta_n)r(\zeta_n)\left<\eta_n,\zeta_n\right>)^2}{|\varphi(\eta_n)-\varphi(\zeta_n)|^2\cdot r^2(\eta_n)}-1\\&
\frac{r(\eta_n)^2r(\zeta_n)^2(\left<\eta_n,\zeta_n\right>^2-1)}{r^2(\eta_n)|\varphi(\eta_n)-\varphi(\zeta_n)|^2}\\&\le \frac{r(\eta_n)^2r(\zeta_n)^2(\left<\eta_n,\zeta_n\right>^2-1)}{2r^2(\eta_n)L^2(1-\left<\eta_n,\zeta_n\right>)}\le -r^2_{\min}/L^2
\end{split}\]

Thus $\cos\alpha(x_n,y_n)\le \sqrt{1-r^2_{\min}/L^2}$ which is contradiction, because $\cos\alpha(x_n,y_n)$ tends to $1$.
  This proves our proposition. 
\end{proof}

\begin{remark}\label{alfabeta}
It can be proved that if the domain satisfies the $\beta$-cone condition almost everywhere (or on some dense subset of the  boundary), then it satisfies the $\beta$-cone condition everywhere. On the other hand there exists a domain satisfying the $\alpha$-tangent condition almost everywhere but not everywhere.
For example, consider the domain $D$ bounded by the graph of the
Cantor step function $\mathcal{C}$ and the following segments
$[T,A]$, $[A,B]$, $[B,C]$, and $[C,O]$, where $T=(1,1)$, $A=(2,1)$, $B=(2,-1)$, $C=(0,-1)$, $O=(0,0)$  (Figure \ref{ct}). The domain
$D$ is strictly starlike w.r.t. the point
$S=(1,-1/2)$ and satisfies the $\alpha$-tangent condition almost everywhere but not
the $\beta$-cone condition almost everywhere. A point where the cone condition is ruined is
$T=(1,1)$ together with some points in its neighborhood.
The size of the neighborhood depends on a given positive number $\beta$.
\end{remark}

\begin{figure}[htp]\label{1}
\centering
\includegraphics[width=.32 \textwidth]{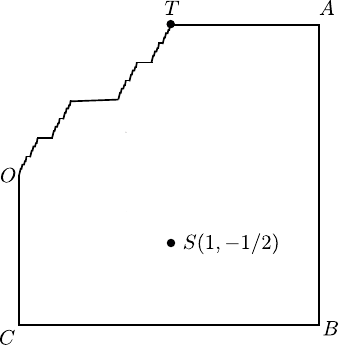}
\caption{\label{ct} A domain satisfies the $\alpha$-tangent condition almost everywhere with $\alpha=\arctan1/2$, but not everywhere.}
\end{figure}

\begin{nonsec}{\bf The chordal metric.}
The chordal metric is defined by
\beq\label{q}
\left\{\begin{array}{ll}
q(x,y)=\frac{|x-y|}{\sqrt{1+|x|^2}\sqrt{1+|y|^2}},&\,\,\, x, y\neq\infty,\\
q(x,\infty)=\frac{1}{\sqrt{1+|x|^2}},&\,\,\, x\neq\infty.
\end{array}\right.
\eeq
\end{nonsec}

\begin{nonsec}{\bf M\"obius metric {\rm (Seittenranta \cite{s2})}.}
Let $G$ be an open subset of $\overline{\R}^n$ with card $\partial G\geq 2$. For all $x,\,y\in G$, the
M\"obius (or absolute ratio) metric $\delta_G$ is defined as
\begin{eqnarray*}
\delta_G(x,y)=\log (1+\sup_{a,b\in\partial G}|a,x,b,y|)\,,\quad
|a,x,b,y|= \frac{|a-b||x-y|}{|a-x||b-y|} \,.
\end{eqnarray*}
It is a well-known basic fact that $\delta_G$ agrees
with the hyperbolic metric both in the case of the unit ball as well
as in the case of the half space (cf. \cite[8.39]{vu}).
\end{nonsec}

\begin{nonsec}{\bf Ferrand's metric.}
Let $G\subset\overline{{\mathbb R}}^n\,$  be a domain
with card $\partial G\geq 2$. We define a continuous density function
$$w_G(x)=\sup_{a,b\in\partial G}\frac{|a-b|}{|x-a||x-b|}\,,\,\,\,\,x\in G\setminus\{\infty\},$$
and a metric $\sigma$ in $G$,
$$\sigma_G(x,y)=\inf_{\gamma\in \Gamma}\int_{\gamma}w_G(t)|dt|,$$
where $\Gamma$ is the family of all rectifiable curves joining $x$
and $y$ in $G$.
\end{nonsec}

\begin{proposition}\label{pr}
 The chordal metric $q$ is invariant under the inversion in the unit sphere. The M\"obius metric $\delta_G$ and Ferrand's metric $\sigma_G$
are M\"obius-invariant.
\end{proposition}

\begin{lemma}\cite[Theorem 3.16]{s2}\label{delta}
Let $G\,,G'\subsetneq\RN$ be open sets and $f: G\rightarrow G'$ an
$L$-bi-Lipschitz map w.r.t. the Euclidean metric. Then $f$ is an
$L^4$-bi-Lipschitz map  w.r.t. the M\"obius  metric.
\end{lemma}

\begin{lemma}\cite[Theorem 3.18]{s2}\label{sigma}
Let $G\,,G'\subsetneq\RN$ be domains and $f: G\rightarrow G'$ an
$L$-bi-Lipschitz map w.r.t. the Euclidean metric. Then $f$ is an
$L^4$-bi-Lipschitz map  w.r.t. Ferrand's metric.
\end{lemma}

Without a proof, the following lemma is given in \cite[1.34]{s1}.

\begin{lemma}\label{lees}
Let $f: \overline{\R}^n\rightarrow\overline{\R}^n$ with $f(0)=0$ and $f(\infty)=\infty$.

(1) If $f|_{\RN}$ is an $L$-bi-Lipschitz map w.r.t.
the Euclidean metric, then $f$ is an $L^3$-bi-Lipschitz
map w.r.t. the chordal metric.

(2) If $f$ is an $L$-bi-Lipschitz map w.r.t.
the chordal metric, then $f|_{\RN}$ is an
$L^3$-bi-Lipschitz map w.r.t. the Euclidean metric.
\end{lemma}
\begin{proof}
(1) If $f:\R^n\rightarrow\R^n$ is an $L$-bi-Lipschitz mapping w.r.t.
the Euclidean metric and $f(0)=0$, then
\beq\label{fL} \frac 1L |x-y|\leq |f(x)-f(y)|\leq L|x-y| \eeq and
\beq\label{fLinfty} \frac 1L |x|\leq |f(x)|\leq L|x|.
\eeq
For $x,y\in\R^n$ and $x\neq y$, we have
$$\frac{q(f(x),f(y))}{q(x,y)}=
\frac{|f(x)-f(y)|}{|x-y|}\cdot\sqrt{\frac{1+|x|^2}{1+|f(x)|^2}}
\cdot\sqrt{\frac{1+|y|^2}{1+|f(y)|^2}},$$
by (\ref{fL}) and (\ref{fLinfty}), and hence
$$\frac{q(f(x),f(y))}{q(x,y)}\leq L^3.$$
If $y=\infty$ and $x\neq y$, then
$$\frac{q(f(x),f(\infty))}{q(x,\infty)}=\sqrt{\frac{1+|x|^2}{1+|f(x)|^2}}\leq L.$$
Applying the above argument to $f^{-1}$, we easily get
$$\frac{q(f^{-1}(x),f^{-1}(y))}{q(x,y)}\le {L^3}$$
and hence
$$\frac{q(f(x),f(y))}{q(x,y)}\ge \frac 1{L^3}.$$

(2) If $f: \overline\R^n\rightarrow\overline\R^n$ is an $L$-bi-Lipschitz mapping w.r.t.
the chordal metric and $f(\infty)=\infty$, then \beq\label{qfL}
\frac 1L q(x,y)\leq q(f(x),f(y))\leq L q(x,y) \eeq
and
\beq\label{qfLinfty}
\frac 1L q(x,\infty)\leq q(f(x),f(\infty))\leq
L q(x,\infty).
\eeq
For $x,y\in\R^n$ and $x\neq y$, we have
$$\frac{|f(x)-f(y)|}{|x-y|}=\frac{q(f(x),f(y))}{q(x,y)}\cdot\frac{q(x,\infty)}{q(f(x),f(\infty))}\cdot\frac{q(y,\infty)}{q(f(y),f(\infty))},$$
by (\ref{qfL}) and (\ref{qfLinfty}), and hence
$$\frac{|f(x)-f(y)|}{|x-y|}\leq L^3.$$
Applying the above argument to $f^{-1}$, we finally get
$$\frac{|f(x)-f(y)|}{|x-y|}\geq\frac 1{L^3}.$$
This completes the proof.
\end{proof}



\section{The radial projection to the unit sphere}
Let $\mathrm{dist}(\mathcal{A},a)$ be the distance from the point $a$ to the set $\mathcal{A}$. The radial projection in \eqref{rpx} maps a point $z\in \RN\setminus\{0\}$ to
$$\Pi(z)=z^*; \,\,\,z^*=z/|z|.$$
In this section, we mainly study the Lipschitz properties of this projection.

J. Luukkainen and J. V\"ais\"al\"a obtained
$$|\Pi(x)-\Pi(y)|\le |x-y|/\sqrt{|x||y|},$$
for $x, y\in \RN\setminus\{0\}$(\cite[Lemma 2.12]{lv}). The following lemma improves this inequality.

\begin{lemma}\label{pole}
For two distinct points $x, y\in \RN\setminus\{0\}$, there holds
\begin{equation}\label{bine}
\frac{|x-y|}{|x^*-y^*|}\ge \frac{|x|+|y|}{2},
\end{equation}
 the equality is attained if and only if $|x|=|y|$.
\end{lemma}
\begin{proof} Without loss of generality we may assume that $0$, $x$ and $y$ are three distinct complex numbers.
Let $x=pe^{is}$ and $y=qe^{it}$ where $p,q>0$, $s,t\in[0,2\pi]$, and $s\neq t$. Then
$$
\left(\frac{|x-y|}{|x^*-y^*|}\right)^2-\left(\frac{|x|+|y|}{2}\right)^2=\frac{1}{4} (p - q)^2 \cot^2\left(\frac{t-s}{2}\right)\ge 0,
$$
which implies the inequality \eqref{bine}. The equality case is clear.
\end{proof}

By Lemma~\ref{pole} we immediately have the following corollaries.
\begin{corollary}\label{co111}
Let $\mathcal{A}$ be any nonempty subset of $\RN$ with $\mathrm{dist}(\mathcal{A},0)>0$. For two distinct points $x, y\in \mathcal{A}$ we have
$$\frac{|x-y|}{|x^*-y^*|}\ge \mathrm{dist}(\mathcal{A},0).$$
\end{corollary}

\begin{corollary}\label{co1}
Let $\mathcal{A}$ be any nonempty subset of $\RN$ with $\mathrm{dist}(\mathcal{A},0)>0$. Let $\Pi:\mathcal{A}\rightarrow S^{n-1}$ be the radial projection. Then $\Pi$ is Lipschitz continuous and
\begin{equation}\label{clos}\mathrm{Lip}(\Pi)\le  \frac{1}{\mathrm{dist}(\mathcal{A},0)}.
\end{equation}
\end{corollary}

Our next aim is to study the equality case in \eqref{clos}, under suitable
conditions on $\mathcal{A}$. Fix a connected closed set
$\mathcal{A} \subset \mathbb{R}^n \setminus \{ 0\}$ with ${\rm dist}(\mathcal{A},0)>0\,.$
Let $N(\mathcal{A})=\{h\in\mathcal{A}: {\rm dist}(\mathcal{A},0)=|h|\}$.
We say that
 $\mathcal{A}$ is \textbf{admissible}
if there exist two sequences of  points $\{x_k\}, \{y_k\}\subset\mathcal{A}$, such that $x_k\neq y_k$, tending to two points $h$ and $h'$, respectively, $h,\,h'\in N(\mathcal{A})$ and satisfying one of the following conditions:\\
(i) $h\neq h'$;\\
(ii) $h=h'$ and
 \begin{equation}\label{app}
 |x_k|-|y_k|={\rm o}(\angle(x_k,0,y_k)),\,\,\text{ $k\to \infty$}.
\end{equation}

\begin{example}\label{eg}
The important examples of admissible subsets are the boundaries of domains in $\RN$ strictly starlike w.r.t. the origin. Let $\mathcal{M}$ be one of these boundaries. Because $\mathcal{M}$ is compact, we may assume that
$h\in \mathcal{M}$ which is one of the closest points from the origin. In fact there exist two sequences of points $\{x_k\}, \{y_k\}\subset\mathcal{M}$ , such that for $k\ge 1$, $x_k\neq y_k$ and $|x_k|=|y_k|$, which converge to the same point $h$ and satisfy \eqref{app}. We consider $\mathcal{M}$ in two cases.
The existence of such sequences is trivial if $\mathcal{M}$ coincides with a sphere centered at origin. Otherwise we consider the function $\phi:\mathcal{M}\to\mathbb{R}^+$ by $\phi(x)=|x|$. Then there is a $\delta>0$ such that $[|h|,|h|+\delta]\subset \phi(\mathcal{M})$. This means that  for $t\in(|h|,|h|+\delta)$ there exists $x\in \mathcal{M}$ such that $\phi(x)=t$. The subset $\phi^{-1}({t})=\mathcal{M}\cap S^{n-1}(0,t)$ of $\mathcal{M}$ contains an element $y$ different from $x$. This case is trivial for $n\ge 3$, since $n-1\ge 2$, and deleting a point, we cannot ruin the connectivity of the set. If $n-1=1$, then we use an additional argument that $\mathcal{M}$ is a closed curve, which means that it cannot be separated by deleting only one point. Thus for all dimensions $n\ge2$ and all $\mathcal{M}\subset \mathbb{R}^n$ as above,
there exist different sequences of points $\{x_k\}, \{y_k\}\subset\mathcal{M}$ satisfying $|x_k|=|y_k|$. This means that they satisfy \eqref{app}.
\end{example}

We offer a counterexample.
\begin{example}
Let $\mathcal{A}\subset\R^2$ be the union of the unit circle $S^1$ and the interval $[1/2,1]$.
Then $h=1/2\in N(\mathcal{A})$, and if two sequences of distinct points $\{x_k\}, \{y_k\}\subset \mathcal{A}$ converge to $1/2$, then $\angle(x_k,0,y_k)=0$ as $k\rightarrow \infty$, which means that they do not satisfy \eqref{app}.
Further, we consider the radial projection $\Pi: \mathcal{A}\rightarrow S^1$. For two distinct points $x ,y\in \mathcal{A}$, we have
\begin{eqnarray*}
\frac{|\Pi(x)-\Pi(y)|}{|x-y|}=\left\{\begin{array}{ll}1,&\,\,\, x, y\in S^1,\\
0,&\,\,\, x,y\in[1/2,1].
\end{array}\right.
\end{eqnarray*}
If $x\in[1/2,1]$ and $y\in S^1\backslash\{1\}$, by symmetry we may assume that $y=e^{i \theta}(0<\theta\le\pi)$. Then
\begin{eqnarray*}
\max_x\left\{\frac{|\Pi(x)-\Pi(y)|}{|x-y|}\right\}=\left\{\begin{array}{ll}\frac{|e^{i \theta}-1|}{\sin\theta},&\,\,\, \theta\in(0,\pi/3],\vspace{1mm}\\
\frac{|e^{i \theta}-1|}{|e^{i \theta}-1/2|},&\,\,\, \theta\in(\pi/3,\pi].
\end{array}\right.
\end{eqnarray*}
By calculation, we have
$$\frac{|e^{i \theta}-1|}{\sin\theta}=\sqrt{\frac{2}{1+\cos\theta}}\le 2/\sqrt 3 \,\,\, {\rm if} \,\,\,\theta\in(0,\pi/3]$$
and
$$\frac{|e^{i \theta}-1|}{|e^{i \theta}-1/2|}=\sqrt{2-\frac{2}{5-4\cos\theta}}\le 4/3\,\,\, {\rm if}\,\,\,\theta\in(\pi/3,\pi].$$
Therefore, $${\rm Lip}(\Pi)=4/3<2=1/{\rm dist}(\mathcal{A}, 0).$$
\end{example}

\begin{lemma}\label{le113}
Let $\mathcal{A}$ be a nonempty subset of $\RN$ with $\mathrm{dist}(\mathcal{A},0)>0$. Let $\Pi:\mathcal{A}\rightarrow S^{n-1}$ be the radial projection. Then ${\rm Lip}(\Pi)=1/\rm{dist}(\mathcal{A},0)$ if and only if $\mathcal{A}$ is admissible.
\end{lemma}
\begin{proof}
Assume that $\mathcal{A}$ is admissible and we consider two cases.

Case 1. There exist two sequences of distinct points $\{x_k\}, \{y_k\}\subset\mathcal{A}$ converging to two different points $h$, $h'$, respectively, with $|h|=|h'|=\mathrm{dist}(\mathcal{A},0)$, i.e. $\lim_{k\to \infty}x_k=h$ and $\lim_{k\to \infty}y_k=h'$. Then we have $$\lim_{k\to \infty}\frac{|x^*_k-y^*_k|}{|x_k-y_k|}=\frac{|h^*-{h'}^*|}{|h-h'|}=\frac{|h^*|}{|h|}=\frac{1}{\mathrm{dist}(\mathcal{A},0)}.$$

Case 2. There exist two sequences of distinct points $\{x_k\}, \{y_k\}\subset\mathcal{A}$ converging to the same point $h$, i.e.
$$\lim_{k\to \infty}x_k=h= \lim_{k\to \infty}y_k.$$
In the notation of the proof of Lemma~\ref{pole}, we assume that the points $0,x_k,y_k$ belong  to the complex plane, too.
Let $x_k=p_ke^{is_k}$, $y_k=q_ke^{it_k}$ where $p_k,q_k>0$, $s_k,t_k\in[0,2\pi]$, and $s_k\neq t_k$. By \eqref{app} we have
\begin{equation}\label{sp}
\left(\frac{|x_k-y_k|}{|x^*_k-y^*_k|}\right)^2-\left(\frac{|x_k|+|y_k|}{2}\right)^2
=\frac{1}{4} (p_k - q_k)^2 \cot^2\left(\frac{t_k-s_k}{2}\right)\rightarrow 0,\,\,\,k\to \infty,
\end{equation}
which implies
$$\lim_{k\to \infty}\frac{|x^*_k-y^*_k|}{|x_k-y_k|}=\frac{1}{\mathrm{dist}(\mathcal{A},0)}.$$
By Case 1, Case 2 and Corollary \ref{co1}, we get $\mathrm{Lip}(\Pi)=1/\mathrm{dist}(\mathcal{A},0)$.

\medskip

Assume now that $\mathrm{Lip}(\Pi)=1/\mathrm{dist}(\mathcal{A},0)$.
  Choose two sequences of distinct points
  $\{x_k\}, \{y_k\}\subset \mathcal{A}$ such that $$\mathrm{Lip}(\Pi)=\lim_{k\to \infty}\frac{|x^*_k-y^*_k|}{|x_k-y_k|}.$$
By \eqref{bine}, we have
\begin{equation}\label{jep}
\lim_{k\to \infty}\frac{|x^*_k-y^*_k|}{|x_k-y_k|}\le \limsup_{k\to \infty}\frac{2}{|x_k|+|y_k|}\le 1/\mathrm{dist}(\mathcal{A},0).
\end{equation}
and hence
\beq\label{xyk}
\limsup_{k\to \infty}\frac{2}{|x_k|+|y_k|}=1/\mathrm{dist}(\mathcal{A},0)=\lim_{k\to \infty}\frac{|x^*_k-y^*_k|}{|x_k-y_k|}.
\eeq
There exists a subsequence still denoted by $\{x_k\}$ tending to a point $h$ such that $\mathrm{dist}(\mathcal{A},0)=|h|$.
Similarly, $\{y_k\}$ tends  to a point $h'$ with $\mathrm{dist}(\mathcal{A},0)=|h'|$. If $h=h'$, by \eqref{sp} the sequences $\{x_k\}$ and $\{y_k\}$ satisfy the relation \eqref{app}. Thus $\mathcal{A}$ is admissible.
\end{proof}
\comment{
\begin{remark}\label{pams}
Let $\mathcal{A}\subset \mathbb{R}^n$ be a subset containing at least two points and $a\notin {\mathcal{A}}$. Let $\Pi_a:\mathcal{A}\to S^{n-1}$ be the mapping defined by
$$\Pi_a(x)=\frac{x-a}{|x-a|}.$$ Then $\Pi_a$ is Lipschitz if  $a\notin\overline{\mathcal{A}}$ by Corollary \ref{co1}. Moreover the Lipschitz constant of $\Pi_a$
satisfies $$\mathrm{Lip}(\Pi_a)\le \frac{1}{\mathrm{dist}(\mathcal
A,a)}.$$
In particular, if $\mathcal{A}\subset\RN$ is admissible after a translation $f(x)=x-a$, {\bf THE REST OF THE SENTENCE IS NOT CLEAR} such as $\mathcal{A}$ is the boundary of a strictly starlike domain w.r.t. the point $a$, then the previous equality holds.
\end{remark}
}

O. Martio and U. Srebro defined the so-called \emph{radial stretching} which in fact is the inverse of the radial extension $\varphi_1$ in \eqref{re}. They proved that
\begin{lemma}\label{ms1}\cite[Lemma 2.4]{mar2}
Let $\mathcal{M}$ be the boundary of a strictly starlike domain w.r.t. the origin and $\varphi$ be the inverse map of the radial projection to the unit sphere. If $\varphi$ is bi-Lipschitz, then $\varphi_1$ is bi-Lipschitz.
\end{lemma}

\begin{lemma}\label{ms2}\cite[Lemma 2.7]{mar2}
If the domain bounded by $\mathcal{M}$ satisfies the cone condition for some $\beta\in(0,\pi/4]$, then $\varphi_1$ is bi-Lipschitz.
\end{lemma}

By Lemma \ref{ms1}, it is clear that $\varphi_1$ is bi-Lipschitz if and only if $\varphi$ is bi-Lipschitz.
\begin{theorem}\label{multi}
Let $\mathcal{M}$ be the boundary of a domain in $\RN$, strictly starlike w.r.t. the origin. Let $\varphi: S^{n-1}\to \mathcal{M}$ be a homeomorphism which sends $R\cap S^{n-1}$ to $R\cap \mathcal{M}$, where $R$ is the ray from $0$. Then
\begin{equation}\label{terza}
{{\mathrm{Lip}(\varphi^{-1})}}=\frac{1}{\mathrm{dist}(\mathcal{M},0)}=
\limsup_{r\to 0}\sup_{|z-\zeta|<r}\frac{|z-\zeta|}{|\varphi(z)-\varphi(\zeta)|}.
\end{equation}
Moreover, $\varphi$ is bi-Lipschitz if and only if the domain bounded by $\mathcal{M}$ satisfies the tangent condition for some $\alpha$ and if by $\alpha=\alpha(\mathcal{M})$ we denote the maximum of all such constants, then we have
\begin{equation}\label{prima}
{\mathrm{Lip}(\varphi)}=\sup_{\zeta}\limsup_{z\to
\zeta}\frac{|\varphi(z)-\varphi(\zeta)|}{|z-\zeta|}
\end{equation}
and for $r=|x|$,
\begin{equation}\label{qinka}
{\mathrm{Lip}(\varphi)}={\rm ess}\sup_{x\in\mathcal{M}}\frac{|x|}{\sin\alpha(x)}\in\left[\frac{r_{\mathrm{ min}}}{\sin\alpha}, \frac{r_{\mathrm{max}}}{\sin\alpha}\right].
\end{equation}
\end{theorem}
\begin{proof}
By Example \ref{eg}, Lemma~\ref{le113} and \eqref{xyk}, there exist two sequences of distinct points $\{x_k\}, \{y_k\}\subset\mathcal{M}$ such that
$\lim_{k\to \infty}x_k=\lim_{k\to \infty}y_k=h$ with $|h|=\mathrm{dist}(\mathcal{M},0)$ and
$$
\frac{1}{\mathrm{dist}(\mathcal{M},0)}=\lim_{k\to \infty}\frac{|x^*_k-y^*_k|}{|x_k-y_k|}
$$
Then
$${{\mathrm{Lip}(\varphi^{-1})}}=\lim_{k\to \infty}\frac{|x^*_k-y^*_k|}{|x_k-y_k|}\le
\limsup_{r\to 0}\sup_{|z-\zeta|<r}\frac{|z-\zeta|}{|\varphi(z)-\varphi(\zeta)|}
\le{{\mathrm{Lip}(\varphi^{-1})}}.$$
Hence \eqref{terza} holds.

By Proposition \ref{stat} and Lemma \ref{ms2}, or by \cite[Eq.~5.11]{gv}, it is clear that $\varphi$ is bi-Lipschitz if the domain bounded by $\mathcal{M}$ satisfies the $\alpha$-tangent condition.

On the other hand, we suppose that $\varphi$ is bi-Lipschitz, then ${\mathrm{Lip}(\varphi)}<\infty$. For two distinct points $x, z\in\mathcal{M}$. Let $\angle(z,0,x)=\theta$, then by the Law of Sines,
\[\begin{split}\sin\angle(z,x,0)&=\frac{|z|\sin\theta}{|z-x|}\\&=\frac{|\varphi^{-1}(z)-\varphi^{-1}(x)|}{|z-x|}|z|\sin\left(\frac{\pi-\theta}{2}\right)\\&\ge \frac{{\rm dist}(\mathcal{M}, 0)}{\rm{Lip}(\varphi)}\sin\left(\frac{\pi-\theta}{2}\right).\end{split}\]
Hence $$\liminf_{z\rightarrow x}\sin\angle(z,x,0)\ge \frac{{\rm dist}(\mathcal{M}, 0)}{\rm{Lip}(\varphi)}.$$
By Corollary \ref{co111} we have $\rm{Lip}(\varphi)\ge {\rm dist}(\mathcal{M}, 0)$
which implies that there exists $\alpha\in(0, \pi/2]$ such that $\alpha(x)=\liminf_{z\rightarrow x}\angle(z,x,0)\ge\alpha$.

The equality \eqref{prima} follows from \eqref{limipo}.

The equality \eqref{qinka} follows from \eqref{limipo1} and \eqref{rdot}.
\end{proof}

\begin{lemma}\label{ar1}\cite[Theorem 3.1]{comp1}
Let $\gamma$ be the curve which is the boundary of a domain in $\R^2$ strictly starlike w.r.t. the origin.
Let $\phi(t)=r(t) e^{it}: [0,2\pi]\rightarrow \gamma$ and $\varphi(e^{it})=\phi(t)$. Then
\begin{equation}\label{limipo}
\sup_{t\neq s}\frac{|\phi(t)-\phi(s)|}{|e^{it}-e^{is}|}=\sup_s\limsup_{t\to s}
\frac{|\phi(t)-\phi(s)|}{|e^{it}-e^{is}|}.
\end{equation}
Moveover, if $\sup_s\limsup_{t\to s}\frac{|\phi(t)-\phi(s)|}{|e^{it}-e^{is}|}<\infty,$  then
\begin{equation}\label{limipo1}
{\rm Lip}(\varphi)={\rm Lip}(\phi)=\mathrm{ess}\sup_{t}\sqrt{ r'^2(t)+r^2(t)}.
\end{equation}
\end{lemma}
In \cite[Theorem 3.1]{comp1} a condition similar to the one
in Lemma \label{ar1} was studied under the additional assumption that $\phi$ be Lipschitz. However, this assumption is redundant by Lemma \ref{lel}.
\begin{proof}
If $\sup_s\limsup_{t\to s}
\frac{|\phi(t)-\phi(s)|}{|e^{it}-e^{is}|}=\infty$, then it is trivial that \eqref{limipo} holds. We now consider that $\sup_s\limsup_{t\to s} \frac{|\phi(t)-\phi(s)|}{|e^{it}-e^{is}|}<\infty$. Then
$$\sup_s\limsup_{t\to s} \frac{|\phi(t)-\phi(s)|}{|e^{it}-e^{is}|}=\sup_s\limsup_{t\to s} \frac{|\phi(t)-\phi(s)|}{|t-s|}<\infty.$$
Let $\phi(t)=(\phi_1(t),\phi_2(t))$. By Lemma \ref{lel} $\phi_1, \phi_2$ are both Lipschitz continuous and hence $\phi$ is a Lipschitz map. Then by \cite[Theorem 3.1]{comp1}, the equality \eqref{limipo} holds. We also have
$${\rm Lip}(\phi)\le{\rm Lip}(\varphi)=\sup_s\limsup_{t\to s}\frac{|\varphi(e^{it})-\varphi(e^{is})|}{|e^{it}-e^{is}|}=\sup_s\limsup_{t\to s} \frac{|\phi(t)-\phi(s)|}{|t-s|}\le {\rm Lip}(\phi)\,.$$
This inequality together with \eqref{bil} yields \eqref{limipo1}.
\end{proof}

\begin{lemma}\label{lel}
 Assume that $g:[0,1]\to \R$ is a real function such that $$M=\sup_t\limsup_{s\to t}\frac{|g(s)-g(t)|}{|t-s|}<\infty,$$ then $g$ is $M$-Lipschitz continuous.
\end{lemma}
\begin{proof}
 Suppose that $E$ is a measurable subset of an interval, $g$ is a function on $E$ such that for every $x\in E$, $|D^+(g)(x)|\le M$, where $$|D^+(g)(x)|=\limsup_{y \to x}\frac{|g(x)-g(y)|}{|x-y|}.$$ Then
$$\mathrm{mes}(g(E))\le M\mathrm{mes}(E),$$
see e.g. the proof of Lemma 3.13 in \cite{misha}. Here $\mathrm{mes}$  is the (outer) Lebesgue measure.
Note that  \cite[Lemma~3.13]{misha} assumes that $g$ is differentiable, but the proof uses only the upper bound $|D^+(g)(x)|\le M$ for all $x$. In our case, $E$ is an interval $[s,t]$ in $[0,1]$ and
$$M=\sup_x\limsup_{y\to x}\frac{|g(x)-g(y)|}{|x-y|}.$$
Then for every $s,t\in[0,1]$, we have
$$|g(s)-g(t)|\le M\mathrm{mes}(g([s,t]))\le M|s-t|.$$
Hence $g$ is $M$-Lipschitz.
\end{proof}

By (\ref{terza}), we have
\begin{corollary}\label{arush}
Let $\mathcal{M}$ be the boundary of a domain in $\RN$, strictly starlike w.r.t. the origin. Let $\varphi$ be as in Theorem \ref{multi}. Then
$$\liminf_{r\to 0}\inf_{|z -\zeta|<r}\frac{|\varphi(z)-\varphi(\zeta)|}{|z-\zeta|}= {\inf_{z,\zeta\in
S^{n-1}}\frac{|\varphi(z)-\varphi(\zeta)|}{|z-\zeta|}}=\mathrm{dist}(\mathcal{M},0).$$
\end{corollary}
\medskip

\section{bi-Lipschitz and quasiconformal constants for radial extensions}

For the statement of our results we carry out some preliminary considerations.
Let $\gamma$ be the boundary of a domain in $\R^2$ strictly starlike w.r.t. the origin which satisfies the $\alpha$-tangent condition. We will recall some properties
of $\gamma$. Let $t \mapsto r(t)e^{it}$ be the polar parametrization of
$\gamma$.
If the curve $\gamma$ is smooth, following the notations in \cite{comp}, the angle $\lambda_{t}$
between $\zeta=r(t)e^{it}$ and the positive oriented tangent at $\zeta$
satisfies
\begin{equation}\label{rdot}
\cot\lambda_{t}=\frac{r'(t)}{r(t)}.
\end{equation}
Observe that $\alpha(\zeta)=\min\{\lambda_t,\pi-\lambda_t\}$. Moreover we have
$$
0<\alpha_1=\inf_{t} \lambda_t\le \frac{\pi}2\le \sup_t
\lambda_t =\alpha_2<\pi.
$$
Then
\begin{equation}\label{alal}
\alpha=\min\{\alpha_1,\pi-\alpha_2\}.
\end{equation}

Let $z=\rho e^{it}$,$\varphi(e^{it})=r(t)e^{it}$. Let ${f^g}(z)=g(\rho) \varphi(e^{it})$, for some real positive smooth increasing function $g$, with $g(0)=0$ and $g(1)=1$.  By direct calculation,
$$
|{{f^g}}_z|=\frac{1}{2|z|}|g'\cdot r\cdot \rho-i g\cdot(r'+ir)|\,\, \text{ and }\,\, |{{f^g}}_{\bar
z}|=\frac{1}{2|z|}|g'\cdot r\cdot \rho+i g\cdot(r'+ir)|,
$$
by \eqref{rdot}, we have
$$
|{{f^g}}_z|=\frac{r}{2|z|}|g'\cdot \rho+g-i g\cot\lambda_t|\,\, \text{ and }\,\, |{{f^g}}_{\bar
z}|=\frac{r}{2|z|}|g'\cdot \rho-g+i g\cot\lambda_t|.
$$
In order to minimize the constant of quasiconformality we define the function
$$\kappa(g,z):=\mu^2_{{f^g}}(z).$$ Then
$$\kappa(g,z)=\frac{|g'\cdot \rho-g+i g\cot\lambda_t|^2}{
|g'\cdot \rho+g-i g\cot\lambda_t|^2}=\frac{(h-1)^2+\cot^2\lambda_t}{(h+1)^2+\cot^2\lambda_t}, $$
where
$$h(\rho)=\frac{g'(\rho)\rho}{g(\rho)}.$$
Since
$$\kappa(g,z)\le \frac{(h-1)^2+\cot^2\alpha}{(h+1)^2+\cot^2\alpha},$$
we easily see that the derivative of the last expression w.r.t. $h$ is
$$\frac{4 (h - \csc \alpha) (h + \csc \alpha)}{((h+1)^2 + \cot^2\alpha)^2}, $$
which means that the minimum of the expression $$\frac{(h-1)^2+\cot\alpha}{(h+1)^2+\cot\alpha}$$
is attained for $h=\csc\alpha$. Further the unique solution of differential equation
$$\frac{g'(\rho)\rho}{g(\rho)}=\csc\alpha$$
with $g(0)=0$ and $g(1)=1$ is $g(\rho)=\rho^{\csc\alpha}$.
This means that the minimal constant of quasiconformality for radial stretching mappings is attained by the mapping $${\varphi^\circ}(z)=|z|^{\csc\alpha}\varphi(z/|z|).$$
Further
\begin{equation}\label{expl}
\Lambda_{\varphi^\circ}(z)=|\varphi^\circ_z| +|\varphi^\circ_{\bar z}|=\frac{r(t)\rho^{\csc \alpha-1}}{2}(|\csc \alpha +1-i \cot \lambda_t|+|\csc \alpha -1+i \cot \lambda_t|)
\end{equation}
and
\begin{equation}\label{expl2}
\lambda_{\varphi^\circ}(z)=|\varphi^\circ_z| -|\varphi^\circ_{\bar z}|=\frac{r(t)\rho^{\csc \alpha-1}}{2}(|\csc \alpha +1-i \cot \lambda_t|-|\csc \alpha -1+i \cot \lambda_t|)
\end{equation}
or
\begin{equation}\label{david}
\frac{1}{\lambda_{\varphi^\circ}(z)}=\frac{\rho^{1-\csc \alpha}}
{2r(t)\csc \alpha}(|\csc \alpha +1-i \cot \lambda_t|+|\csc \alpha -1+i \cot \lambda_t|).
\end{equation}
This implies that ${\varphi^\circ}$ is quasiconformal with $k=\tan(\frac{\pi}{4}-\frac{\alpha}{2})$ and $K=\cot\frac{\alpha}{2}$.

In general for $a>0$ we define $\varphi_a(z)=|z|^{a}\varphi(e^{it})$ and obtain
\begin{equation}\label{expl0}
\Lambda_{\varphi_a}(z)=\frac{r(t)\rho^{a-1}}{2}(|a +1-i \cot \lambda_t|+|a -1+i \cot \lambda_t|)
\end{equation}
and
\begin{equation}\label{drita}
\frac{1}{\lambda_{\varphi_a}(z)}=\frac{\rho^{1-a}}{2r(t)a}(|a +1-i \cot \lambda_t|+|a -1+i \cot \lambda_t|).
\end{equation}
\medskip

We now formulate the main result of this section.
\begin{theorem}\label{star}
Let $\gamma$ be the boundary of a domain in $\R^2$ strictly starlike w.r.t. the
origin, and with a polar parametrization by a
homeomorphism $\varphi(e^{it})=r(t)e^{it} :S^1\to \gamma$. Let $a>0$ and
$\varphi_a: {\R} ^2\to {\R}^2$ be the radial extension of $\varphi$ with $\varphi_a(z)=|z|^a\varphi(z/|z|)$ and $\varphi_a(0)=0$.
Then
\begin{itemize}
    \item[a)] $\varphi_a$ is bi-Lipschitz if and only if the domain bounded by $\gamma$ satisfies the tangent condition for some $\alpha$ and $a=1$. Moreover,
\begin{eqnarray}\label{L2}
{\rm Lip}(\varphi_1)=L_1=\frac{r_{\mathrm{ max}}}{2}\left(\sqrt{\csc^2\alpha+3}+\sqrt{\csc^2\alpha-1}\right)
\end{eqnarray}
 and
\begin{eqnarray}\label{L1}
{\rm Lip}(\varphi^{-1}_1)=L_2 =\frac{1}{2r_{\mathrm{min}}}\left(\sqrt{\csc^2\alpha+3}+\sqrt{\csc^2\alpha-1}\right).
\end{eqnarray}
Here $r_{\mathrm{ max}}$ and $r_{\mathrm{ min}}$ are the maximum and the minimum of the function $r$, respectively.
If $\gamma$ satisfies the tangent condition for some $\alpha$, and if by $\alpha=\alpha(\gamma)$ we denote the supremum of all such constants, then for the bi-Lipschitz constant we have $L(\alpha)=\max\{L_1,L_2\}$.

We say that $\gamma_n$ tends to a circle $S^1(0,d)$ if $\alpha(\gamma_n)$ tends to $\pi/2$. Then
$$
\lim_{n\to \infty}L(\alpha)=\max\left\{d,1/d \right\}.
$$
    \item[b)] $\varphi_a$ is quasiconformal if and only if the domain bounded by $\gamma$ satisfies the tangent condition for some $\alpha\in(0,\pi/2]$. In this case the constant of quasiconformality is
         \begin{equation}\label{ka}K_a=\frac{1}{4a}\left(\sqrt{(a-1)^2 + \cot^2\alpha}+\sqrt{(1 + a)^2 + \cot^2\alpha}\right)^2.
         \end{equation} The minimal constant of quasiconformality is attained by ${\varphi^\circ}(z)=|z|^{\csc \alpha}\varphi(z/|z|)$ with ${\varphi^\circ}(0)=0$ and for this mapping we have $$k=\tan(\frac{\pi}{4}-\frac{\alpha}{2}),\ \
K=\cot\frac{\alpha}{2}.$$
\end{itemize}
\end{theorem}
\begin{proof}
 a). If $\varphi_a$ is bi-Lipschitz  then we have $a=1$ by \eqref{bil} and \eqref{expl0}. By Lemma \ref{ms1} and Theorem \ref{multi} we conclude that $\varphi_1$ is bi-Lipschitz if and only if the domain bounded by $\gamma$ satisfies the $\alpha$-tangent condition.

 The Lipschitz constants $L_1, L_2$ follow from \eqref{lamd1}, \eqref{lamd2},\eqref{expl0} \eqref{drita}, and \eqref{bil}.

To show $\lim_{n\to \infty} L=\lim_{\alpha\to \frac{\pi}{2}}L=\max\left\{d,{d^{-1}}\right\}$, we use (\ref{rdot}). Then
$$\log(r_n(t))-\log(r_n(0))=\int_{0}^t \cot \lambda_s\, ds.$$
Without loss of generality we may assume that $r_n(0)=\mathrm{dist}(\gamma_n, 0)$.
Then $$r_n(t)\le \mathrm{dist}(\gamma_n, 0)\exp(t\cot \alpha),$$
and therefore
$$\mathrm{dist}(\gamma_n, 0)\le r(t)\le\mathrm{dist}(\gamma_n, 0)\exp(2\pi \cot \alpha)\rightarrow \mathrm{dist}(S^1(0,r), 0)=r,\,\,\, \alpha\rightarrow \frac{\pi}{2}.$$

Thus by \eqref{L2} and \eqref{L1} we have
$$\lim_{\alpha\to \frac{\pi}{2}}L=\max\left\{d,{d^{-1}}\right\}.$$
\comment{\begin{equation}\label{L}
1\le L\le \frac{\exp(\pi \cot
\alpha)}{2\sin\alpha}\left(\sqrt{1-\sin^2\alpha}+\sqrt{1+3\sin^2\alpha}\right).
\end{equation}}

b). If the domain bounded by $\gamma$ satisfies the $\alpha$-tangent condition, then $\varphi_1$ is bi-Lipschitz by a). Let $R(1/r,r)=B^2(r)\setminus\overline B^2(1/r)$, $r>1$. The function $g(z)=|z|^{a-1}$ is locally Lipschitz in $\overline{R(1/r,r)}$ and hence $g$ is ACL and a.e. differentiable in $R(1/r,r)$. Therefore we have that $\varphi_a=g\cdot \varphi_1$ is ACL and a.e. differentiable in $R(1/r,r)$. Moreover, by \eqref{expl0} and \eqref{drita}, for $x\in R(1/r,r)$
$$H(\varphi'_a(x))\le K_a=\frac{1}{4a}\left(\sqrt{(a-1)^2 + \cot^2\alpha}+\sqrt{(1 + a)^2 + \cot^2\alpha}\right)^2,$$
which implies that $\varphi_a$ is
$K_a$--quasiconformal in $R(1/r,r)$. Letting $r\to \infty$, we see that $\varphi_a$ is $K_a$-quasiconformal in $\R^2\setminus\{0\}$. Since an isolated boundary point is removable singularity, we obtain that $\varphi_a$ is $K_a$-quasiconformal in $\R^2$.

We now prove the reverse implication.
We have to  show that if $\varphi_a$ is $K_a$-quasiconformal then the domain bounded by $\gamma$ satisfies the $\alpha$-tangent condition.
We know that if $\varphi_a$ is $K_a$-quasiconformal, then $\varphi_a$ is ACL and differentiable a.e. and hence $\phi(t)=\varphi(e^{it})$ is absolutely continuous and a.e. differentiable. By virtue of \eqref{expl0} and \eqref{drita}, the domain  bounded by $\gamma$ satisfies the $\alpha$-tangent condition a.e., i.e. for a.e. $t$
\begin{eqnarray*}
\frac{1}{4a}\left(\sqrt{(a-1)^2 + \cot^2\lambda_t}\right.&+&\left.\sqrt{(1 + a)^2 + \cot^2\lambda_t}\right)^2\\
&\le& K_a= \frac{1}{4a}\left(\sqrt{(a-1)^2 + \cot^2\alpha}+\sqrt{(1 + a)^2 + \cot^2\alpha}\right)^2
\end{eqnarray*}
implying that $\lambda_t\ge \alpha$ for a.e. $t$. Here $\alpha$ is chosen so that formula \eqref{ka} holds. Such a value $\alpha>0$ exists because $$\lim_{\alpha\to 0}\frac{1}{4a}\left(\sqrt{(a-1)^2 + \cot^2\alpha}+\sqrt{(1 + a)^2 + \cot^2\alpha}\right)^2=\infty$$
and
$$\lim_{\alpha\to \pi/2}\frac{1}{4a}\left(\sqrt{(a-1)^2 + \cot^2\alpha}+\sqrt{(1 + a)^2 + \cot^2\alpha}\right)^2=\frac{1}{4a}(|1-a|+|1+a|)^2.$$ At this stage we can end the proof by invoking a result of Sugawa \cite[Theorem~1]{su}, however we include detailed self-contained proof. So we prove that $\lambda_t\ge \alpha$ everywhere.
Without loss of generality, we may assume that $\phi(0)=1$ is the non-smooth point and let $\beta=\liminf_{z\to 1}\alpha(z,1)$ (recall that $\alpha(z,1)$ is the acute angle between $z-1$ and $1$). Since $\phi$ is absolutely continuous, we have
$$\phi(t)=1+\int_0^t\psi(s) ds,\,\,\,\,\psi\in L^1([0,2\pi])$$
and
$$\phi'(t)=\psi(t)\,\, a.e.$$
Let $\phi(s)\in\gamma$ with $s\in(0,t)$ be the smooth point and $t\in (0,\pi/2)$. Let $\theta$ be the acute angle between $[0,1]$ and $[\phi(t),1]$. Let $\beta(s)={\rm arg}\, \psi(s)$ and $\lambda_s$ be the acute angle between $[0, \phi(s)]$ and tangent line of $\gamma$ at $\phi(s)$. Then
\begin{equation}\label{oro}
\beta(s)=s+\lambda_s\,\,\,\,\,\,{\rm or}\,\,\,\,\,\,\beta(s)=s+\pi-\lambda_s
\end{equation}
and
\begin{equation}\label{cost}
\cos\theta=\frac{|\Re(\phi(t)-\phi(0))|}{|\phi(t)-\phi(0)|}=
\frac{|\int_0^t\cos\beta(s)|\psi(s)|ds|}{\sqrt{\left(\int_0^t\cos\beta(s)|\psi(s)|ds\right)^2+\left(\int_0^t\sin\beta(s)|\psi(s)|ds\right)^2}}.
\end{equation}
Now we consider the quantity
$$A(t)=\frac{|\int_0^t\sin\beta(s)|\psi(s)|ds|}{|\int_0^t\cos\beta(s)|\psi(s)|ds|}.$$
By \eqref{oro} we have  $\sin\beta(s)=\sin(\lambda_s+s)$ or $\sin\beta(s)=\sin(\lambda_s-s)$ and $\cos\beta(s)=\cos(\lambda_s+s)$ or $\cos\beta(s)=-\cos(\lambda_s-s)$. Then we have
$$A(t)=\frac{|\int_0^t\sin(\lambda_s+\epsilon(s)s)|\psi(s)|ds|}{|\int_0^t\cos(\lambda_s+\varepsilon(s)s)|\psi(s)|ds|},$$
where $\epsilon(s),\varepsilon(s)\in \{-1,1\}$.

We now use the fact that $\lambda_s\ge\alpha$. Then for small enough $t$ (depending on $\alpha$) we obtain
$$A(t)\ge B(t)=\frac{|\int_0^t\sin(\alpha-s)|\psi(s)|ds|}{|\int_0^t\cos(\alpha- s)|\psi(s)|ds|}\ge \frac{|\sin(\alpha-t)|\int_0^t|\psi(s)|ds}{|\cos(\alpha-t)|\int_0^t|\psi(s)|ds}=\tan(\alpha-t).$$
Since
$$\cos\theta=\frac{1}{\sqrt{1+A(t)^2}},$$
we obtain that
$$\liminf_{t\to 0} \cos\theta= \liminf_{t\to 0} \frac{1}{\sqrt{1+A(t)^2}}\le \liminf_{t\to 0} \frac{1}{\sqrt{1+B(t)^2}}=\cos\alpha.$$
Thus $\beta\ge \alpha$ as desired. The proof of the last statement of b) follows from the considerations carried out before the formulation of the theorem.
\end{proof}

We now generalize Theorem~\ref{star} to the $n$-dimensional case.

\begin{theorem}\label{star1}
Let $\mathcal{M}$ be the boundary of a domain in $\R^n$ ($n\ge 3$) strictly starlike w.r.t. the
origin and with a polar parametrization by a
homeomorphism $\varphi(x)=r(x)x :S^{n-1}\to \mathcal{M}$. Let
$\varphi_a: {\R}^n\to {\R}^n$ be the radial extension of $\varphi$, defined by $\varphi_a(x)=|x|^a\varphi(x/|x|)$ i.e. $\varphi_a(x)=|x|^{a-1}R(x)x$, and $\varphi_a(0)=0$, where $R(x)=r(x/|x|)$ is a positive real function. Then

\begin{itemize}
    \item[a)] $\varphi_a$ is bi-Lipschitz if and only if the domain bounded by $\mathcal{M}$ satisfies the tangent condition for some $\alpha\in(0,\pi/2]$,  and $a=1$. Moveover,
\begin{eqnarray}\label{L}
{\rm Lip}(\varphi_1)=L_1=\frac{r_{max}}{2}\left(\sqrt{\csc^2\alpha-1}+\sqrt{\csc^2\alpha+3}\right)
\end{eqnarray}
and
\begin{equation}\label{l1m}
{\rm Lip}(\varphi^{-1}_1)=L_2=\frac{1}{r_{min}}\left(\sqrt{\csc^2\alpha-1}+\sqrt{\csc^2\alpha+3}\right).
\end{equation}
If $\alpha=\alpha(\mathcal{M})$ is the maximum of all positive numbers for which $\mathcal{M}$ satisfies the $\alpha$-tangent condition, then the bi-Lipschitz constant is given by
\begin{equation}\label{qinl}
L(\alpha)=\max\{L_1,L_2\}.
\end{equation}
We say that the sequence $\mathcal{M}_n$ tends to a sphere $S(0,d)$ if $\alpha(\mathcal{M}_n)$ tends to $\pi/2$. In this case
$$\lim_{n\to \infty }L(\alpha)=\max\left\{d,1/d \right\}.$$
\end{itemize}

   \begin{itemize}
    \item[b)]
 $\varphi_a$ is $K_a$-quasiconformal if and only if the domain bounded by $\mathcal{M}$ satisfies the $\alpha$-tangent condition. The constant of quasiconformality is
 $$K_a=\frac{1}{4a}\left(\sqrt{(a-1)^2 + \cot^2\alpha}+\sqrt{(1 + a)^2 + \cot^2\alpha}\right)^2.$$
 The minimal constant of quasiconformality is attained for $a=\csc \alpha$ and the mapping ${\varphi^\circ}(z)=|z|^{\csc \alpha}\varphi(z/|z|)$ with ${\varphi^\circ}(0)=0$ is $K$-quasiconformal with
 $$K=\cot\frac{\alpha}{2}.$$
\end{itemize}
\end{theorem}
\begin{proof}
a) For every $x,y\in \R^n$, the points $0,x,y,\varphi_a(x),\varphi_a(y)$ are in the same plane. Therefore by using the similar argument as in Theorem~\ref{star} a),  we see that the $\alpha$-tangent condition of the domain bounded by $\mathcal{M}$ is equivalent to the bi-Lipschitz continuity of the radial extension $\varphi_1$.

b) 
Assume now that the domain bounded by $\mathcal{M}$ satisfies the $\alpha$-tangent condition. By an argument similar to Theorem~\ref{star} b) we have that $\varphi_a$ is quasiconformal in $R(1/r,r)=B^n(r)\setminus\overline B^n(1/r)\supset \mathcal{M}$. Then it is differentiable a.e. in $R(1/r,r)$.
Let $x=(x_1,\dots,x_n)\in R(1/r,r)\setminus E$, where the Lebesgue measure $|E|=0$, and let $e_k$, $k=1,\dots,n$ denote the standard orthonormal basis  and let us find $\lambda_{\varphi_a}(x)$ and $\Lambda_{\varphi_a}(x)$.
Since $$\Lambda_{\varphi_a}(x)=\sup_{|h|=1}|\varphi'_a(x)h|$$ and $S^{n-1}$ is compact, there exists $h\in S^{n-1}$ such that $\Lambda_{\varphi_a}(x)=|\varphi'_a(x)h|$. Let $\Sigma$ be the $2$-dimensional plane passing through $0,x,h$ and let $\Sigma'$ be another $2$-dimensional plane passing through $0,\varphi_a(x),\varphi'_a(x)h$. Since $\varphi_a(x)=|x|^{a-1}R(x)x$ , we get
$$\varphi'_a(x)h=|x|^{a-3}\left(\left<\mathrm{grad} R(x)|h\right>|x|^2+(a-1)\left<x|h\right>R(x)\right)x+|x|^{a-1}R(x)h,$$
 which implies that $\Sigma'$ can be chosen to be equal to $\Sigma$. Let $T$ be an orthogonal transformation which maps the plane $\C'=\{(x_1,x_2,0,\dots,0):x_1+ix_2\in\mathbb{C}\}\cong\mathbb{C}$ onto $\Sigma$ such that $T(|x|e_1)=x$ and $T(\cos \theta e_1+\sin\theta e_2)=h$. Here $\theta$ satisfies $|x|\cos\theta=\left<x| h\right>$. The mapping $T$ is a linear isometry of $\mathbb{R}^n$. Define
$$\tilde{\varphi_a}(y_1,y_2)=P(T^{-1}\varphi_a(T(y_1,y_2,0,\dots,0))),$$
where $P:\C'\to\mathbb{C}$ is the isometry $P((x_1,x_2,0,\dots,0))=x_1+ix_2$.

Then $\tilde \varphi_a(|x|,0)=P T^{-1}(\varphi_a(x))$ and
 \[\begin{split}
 |\tilde \varphi'_a(|x|,0)|&=\sup_{\beta}|\tilde \varphi'_a(|x|,0)(\cos \beta,\sin\beta )|\\
 &=\sup_{\beta}|(PT^{-1}\cdot \varphi'_a(x)\cdot T)(\cos \beta e_1+\sin\beta e_2)|\le |\varphi'_a(x)|.
 \end{split}\]
By choosing $\beta=\theta$, we see that
$$|(PT^{-1}\cdot \varphi'_a(x)\cdot T)(\cos \beta e_1+\sin\beta e_2)|=|\varphi'_a(x)h|=|\varphi'_a(x)|,$$
which implies $|\tilde \varphi'_a(|x|,0)|=|\varphi'_a(x)|$.
By making use of the proof of two dimensional case \eqref{expl0}, we obtain that
$$\Lambda_{\varphi_a}(x)=\Lambda_{\tilde \varphi_a}((|x|,0))=\frac{r(x)|x|^{a-1}}{2}\left(\sqrt{(a-1)^2 + \cot^2\alpha_{x,h}}+\sqrt{(1 + a)^2 + \cot^2\alpha_{x,h}}\right).$$
Here $\alpha_{x,h}$ is the acute angle between the tangent line on $\mathcal{M}\cap\Sigma$ at $\varphi(x/|x|)$ and the vector $\varphi_a(x)$.

Similar arguments and \eqref{drita} yield that
$$\lambda_{\varphi_a}(x)=\lambda_{\tilde \varphi_a}(|x|,0)=\frac{2ar(x)}{|x|^{1-a}}\left(\sqrt{(a-1)^2 + \cot^2\alpha_{x,h'}}+\sqrt{(1 + a)^2 + \cot^2\alpha_{x,h'}}\right)^{-1},$$
where the angle $\alpha_{x,h'}$ is possibly different from $\alpha_{x,h}$.
Thus we have
\[\begin{split}
H(\varphi'_a(x))=\frac{\Lambda_f(x)}{\lambda_f(x)}&=\frac{1}{4a}\left(\sqrt{(a-1)^2 + \cot^2\alpha_{x,h}}+\sqrt{(1 + a)^2 + \cot^2\alpha_{x,h}}\right)\\
&\cdot\left(\sqrt{(a-1)^2 + \cot^2\alpha_{x,h'}}+\sqrt{(1 + a)^2 + \cot^2\alpha_{x,h'}}\right).
\end{split}\]
Hence $f$ is $K_a$-quasiconformal in $R(1/r,r)$ with  $$K_a=\frac{1}{4a}\left(\sqrt{(a-1)^2 + \cot^2\alpha}+\sqrt{(1 + a)^2 + \cot^2\alpha}\right)^2.$$
Letting $r\to \infty$ , we have that $\varphi_a$ is $K_a$-quasiconformal in $\R^n\setminus\{0\}$ and hence in $\R^n$.

On the other hand, if $\varphi_a$ is $K_a$-quasiconformal in $\R^n$, then by Theorem \ref{star} b) we obtain that the domain bounded by  $\mathcal{M}$ satisfies  the $\alpha$-tangent condition.

For $a=\csc\alpha$, applying the formula of $K_a$ to ${\varphi^\circ}(z)=|z|^{\csc \alpha}\varphi(z/|z|)$, we obtain $K=\cot\alpha/2$, and this is the minimal constant of quasiconformality.

 This completes the proof.
\end{proof}

\begin{example}\label{abc}
a) Let $\mathcal{M}=\partial D$, where $D$ is the cone $\{(x,y,z):z<2-\sqrt{3(x^2+y^2)}, -1< z< 2\}$. Then $\alpha=\pi/6$ and hence ${H}(\varphi^\circ)=2+\sqrt{3} \approx 3.73$.

b) Let $\mathcal{M}=\partial D$, where $D$ is the cylinder $\{(x,y,z):x^2+y^2=1,  -1\le z\le 1\}$. Then $\alpha=\pi/4$ and hence
${H}(\varphi^\circ)=\sqrt{2}+1 \approx 2.41$.

c) Let $\mathcal{M}=\partial{D}$, where $D=[-1,1]^3$ is the cube. Then $\sin\alpha=1/\sqrt{3}$ and hence ${H}(\varphi^\circ)=\sqrt{3}+\sqrt2\approx 3.15$.

d) Let $\mathcal{M}=\partial D$, where $D$ is the ellipsoid $\{(x_1,\dots,x_n):(x_1/a_1)^2+\dots+(x_n/a_n)^2\le 1\}$, $0<a_1\le\dots\le a_n$. We first consider the case of the ellipse $\{(x,y):(x/a)^2+(y/b)^2\le1,0<a<b\}$ whose polar parametrization is
$$r(t)=\frac{ab}{\sqrt{b^2\cos^2t+a^2\sin^2t}},\,\,\,t\in[0,2\pi].$$
By using the argument of symmetry it suffices to consider $t\in[0,\pi/2]$. By \eqref{rdot}, we have
$$0\le\cot\lambda_t=\frac{r'(t)}{r(t)}=\frac{b^2-a^2}{b^2\frac{\cos t}{\sin t}+a^2\frac{\sin t}{\cos t}}\le\frac{b^2-a^2}{2ab}.$$
The equality holds if and only if $\tan t=b/a>1$. Then
$$\cot\frac{\lambda_t} 2=\sqrt{1+\cot^2\lambda_t}+\cot\lambda_t\le\frac ba.$$
Therefore for the ellipsoid the angle $\alpha$ is minimized in the ellipse $(x/a_1)^2+(y/a_n)^2\le1$ and its value is $\alpha=2\arctan(a_1/a_n)$ and ${H}(\varphi^\circ)=a_n/a_1$. Since the linear dilatation of linear mapping $L(x_1,\dots,x_n)=(a_1x_1,\dots,a_nx_n)$ is as well equal to $a_n/a_1$, one may expect that this is the best possible constant of quasiconformality for the ellipsoid ($n\ge 3$). But this is not the case. Concerning this problem we refer to the paper of Anderson \cite{ander}.
\end{example}
\begin{figure}[h]
\subfigure[]{\includegraphics[width=.2 \textwidth]{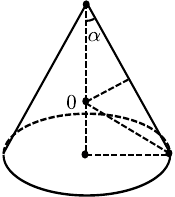}}
\hspace{.1 \textwidth}
\subfigure[]{\includegraphics[width=.15\textwidth]{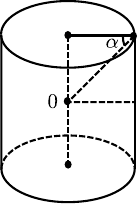}}
\hspace{.1 \textwidth}
\subfigure[]{\includegraphics[width=.2\textwidth]{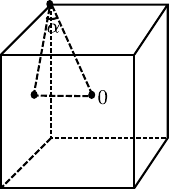}}
\caption{Example \ref{abc}: (a) cone (b) cylinder (c) cube}
\end{figure}

\begin{remark}\label{RR}
In \cite[Lemma~2.7~and~Corollary~2.8]{mar2} the authors prove that the radial stretching mapping is bi-Lipschitz and quasiconformal provided that the domain bounded by $\mathcal{M}$ satisfies the $\beta$-cone condition. Our theorem is somehow optimal, since we have concrete and approximately sharp bi-Lipschitz and quasiconformal constants. Further in \cite[Theorem~5.1]{gv} F.W. Gehring and J. V\"ais\"al\"a obtained some explicit estimates of $K_I$ and $K_O$ for a domain $D$ satisfying the $\alpha$-tangent condition in the three dimensional case. Here $K_I=\inf K_I(f)$ and $K_O=\inf K_O(f)$, where $f$  runs through q.c. mappings of the domain $D$ onto the unit ball $\B^3\subset \R^3$.   Indeed they proved that
\begin{equation}\label{inner}
K_I\le 2^{-1/2}\cot\frac{\alpha}{2}\csc\frac{\alpha}{2},
\end{equation}
\begin{equation}\label{outer}
K_O\le 2^{1/2}\cot\frac{\alpha}{2}\cos\frac{\alpha}{2}.
\end{equation}
Furthermore \eqref{inner} and \eqref{outer} are obtained by making use of a mapping which is in fact the inverse of our mapping ${\varphi^\circ}$ in Theorem~\ref{star1}. This implies that a part of the statement of Theorem~\ref{star1} is not new, at least for three dimensional case. Let $H(f)={\rm ess}\sup_x H(f'(x))$ and $H=\inf H(f)$, where $f$  runs through q.c. mappings of the unit ball $\B^n\subset \R^n$ onto the domain $D$, then we have that
\begin{equation}\label{lineare}
H\le H({\varphi^\circ})=\cot\frac{\alpha}{2}.
\end{equation}
\end{remark}

By \eqref{lineare}, we obtain that
\begin{corollary}
 Let $D\subset\RN$ be the convex domain with $B^n(0,a)\subset D\subset B^n(0,b)$, $0<a<b$. Then $$H\le H(\varphi^\circ)\le\frac{b+\sqrt{b^2-a^2}}{a}<2\cdot\frac{b}{a}.$$
\end{corollary}

\begin{proof}
We can assume without loss of generality that $n=2$. Assume that $$\alpha=\lim_{m\to\infty}\angle(A_m, B_m, 0), $$ where $A_m,B_m\in D$. Since $D$ is convex and contains $B^2(0,a)$, it follows that the line $l=l(A_m,B_m)$ does not intersect the disk $B^2(0,a)$. Let $A_m', B_m'$ be the points in which the line $l$ intersects $S^1(0,b)$. Thus $\alpha\ge \limsup_{m\to\infty}\angle(A'_m, B'_m, 0)$. Since $\angle(A'_m, B'_m, 0)\ge \arcsin \frac{a}{b}$, we obtain that $\cot\frac{\alpha}{2}\le \frac{b+\sqrt{b^2-a^2}}{a}$. This finishes the proof.
\end{proof}
\section{bi-Lipschitz and quasiconformal constants for quasi-inversions}

In this section, we obtain the bi-Lipschitz constants of the quasi-inversion mappings w.r.t. the chordal metric, the M\"obius metric and Ferrand's metric by using the bi-Lipschitz constants of the radial extension maps. We also obtain asymptotically sharp constants of quasiconformality of  quasi-inversions. In order to explain sharpness we make the following definition.
 \begin{definition}
For $t\in[0,1]$, let $\mathcal{M}_t$ be the boundary of a domain in $\RN$, strictly starlike w.r.t. the origin which satisfies the $\alpha_{\mathcal{M}_t}$-tangent condition. Let $\varphi_t:\mathbb{R}^n\to \mathbb{R}^n$ be the radial extension which sends the unit sphere to $\mathcal{M}_t$. We say that $\mathcal{M}_t$ smoothly converges to the sphere $S^{n-1}(r)$, when $t$ goes to $0$, if $\lim_{t\to 0}{\rm ess}\sup_{|x|=1}|\varphi'_t(x)-r\mathbf{I}|=0.$ This in particular means that $\lim_{t\to 0}\alpha_{\mathcal{M}_t}=\pi/2$. Here $\mathbf{I}$ is the identity matrix.
\end{definition}

\begin{lemma}\label{deqin}
Let $a>0$ and $\mathcal{M}$ be the boundary of a domain in $\RN$, strictly starlike w.r.t. the origin and $\varphi: S^{n-1}\to \mathcal{M}$ be the homeomorphism which sends $R\cap S^{n-1}$ to $R\cap \mathcal{M}$, $R$ is the ray from $0$. Let $\varphi_a$ be the radial extension of $\varphi$ and $f_{\mathcal{M}}$ be the quasi-inversion in $\mathcal{M}$. Then $f_\mathcal{M}=\varphi_a\circ f_{S^{n-1}} \circ
\varphi_a^{-1}$.
\end{lemma}
\begin{proof}
It suffices to show that
$f_{\mathcal{M}} \circ \varphi_a=\varphi_a \circ f_{S^{n-1}}.$
For $z\neq 0$ and $z\neq \infty$ we have $$f_\mathcal{M}(\varphi_a(z))=
|\varphi(z/|z|)|^2\frac{\varphi_a(z)}{|\varphi_a(z)|^2}=\frac{\varphi(z/|z|)}{|z|^a}=\varphi_a (f_{S^{n-1}}(z)).$$

If $z=0$ or $z=\infty$, by the definition of $\varphi_a$ and $f_{\mathcal{M}}$ we still have $f_{\mathcal{M}}(\varphi_a(z))=\varphi_a(f_{S^{n-1}}(z)).$
This completes the proof.
\end{proof}

By Theorem \ref{star1}, Lemma \ref{deqin} and \eqref{hf}, we  immediately obtain the following result:
\begin{theorem}\label{khh}
Let $\mathcal{M}$ be the boundary of a domain in $\RN$, strictly starlike w.r.t. the origin which satisfies the $\alpha$-tangent condition. Let $f_\mathcal{M}: \overline{\R^n}\rightarrow\overline{\R^n}$ be the quasi-inversion in $\mathcal{M}$. Then for all $x, y\in \RN\setminus\{0\}$
\begin{eqnarray}\label{ll}
\frac{1}{L^4}\frac{|x-y|}{|x||y|}\leq|f_{\mathcal{M}}(x)-f_{\mathcal{M}}(y)|\leq
L^4\frac{|x-y|}{|x||y|},
\end{eqnarray}
where $L$ is of the form as in \eqref{qinl} of Theorem \ref{star1}.
In particular, if $\mathcal{M}=S^{n-1}$, then \eqref{ll} reduces to the equality
\beq\label{fs}
|f_{S^{n-1}}(x)-f_{S^{n-1}}(y)|=\frac{|x-y|}{|x||y|},
\eeq
which is the same as \eqref{hf} by taking $a=0$ and $r=1$.
\end{theorem}

\begin{theorem}\label{kb}
Let $\mathcal{M}$ be the boundary of a domain in $\RN$, strictly starlike w.r.t. the origin. Let $f_{\mathcal{M}}(z)=\frac{r^2_z z}{|z|^2}$, $z\in\R^n\setminus\{0\}$, be the quasi-inversion in $\mathcal{M}$ and let $x, y\in\R^n\setminus\{0\}$ with $|x|\le|y|$. Then with $\lambda=\frac{|f_{\mathcal{M}}(y)|+|f_{\mathcal{M}}(x)-f_{\mathcal{M}}(y)|}{|f_{\mathcal{M}}(y)|}$ and $z=\lambda x$ we have
\beq\label{qfm1}
|f_{\mathcal{M}}(x)-f_{\mathcal{M}}(z)|\le|f_{\mathcal{M}}(x)-f_{\mathcal{M}}(y)|\le\left(2 r^2_y/r^2_x+1\right) |f_{\mathcal{M}}(x)-f_{\mathcal{M}}(z)|.
\eeq
In particular, if $\mathcal{M}=S^{n-1}(r)$, then
$\lambda=\frac{|x|+|x-y|}{|x|}$ and \eqref{qfm1} reduces to
\beq\label{qfm2}
|f_{S^{n-1}(r)}(x)-f_{S^{n-1}(r)}(z)|\le|f_{S^{n-1}(r)}(x)-f_{S^{n-1}(r)}(y)|\le3|f_{S^{n-1}(r)}(x)-f_{S^{n-1}(r)}(z)|
\eeq
which is the same as in \cite[Lemma 4.5]{bvkv}.
\end{theorem}
\begin{proof}
By calculation, we have
\begin{eqnarray*}
\frac{|f_{\mathcal{M}}(x)-f_{\mathcal{M}}(y)|}{|f_{\mathcal{M}}(x)-f_{\mathcal{M}}(z)|}
=\frac{\lambda}{\lambda-1}\frac{|f_{\mathcal{M}}(x)-f_{\mathcal{M}}(y)|}{|f_{\mathcal{M}}(x)|}
=\frac{|f_{\mathcal{M}}(y)|+|f_{\mathcal{M}}(x)-f_{\mathcal{M}}(y)|}{|f_{\mathcal{M}}(x)|}
\end{eqnarray*}
and
$$1\le\frac{|f_{\mathcal{M}}(y)|+|f_{\mathcal{M}}(x)-f_{\mathcal{M}}(y)|}{|f_{\mathcal{M}}(x)|}\le 2 r^2_y/r^2_x+1.$$
Therefore \eqref{qfm1} follows.

If $\mathcal{M}=S^{n-1}(r)$, then by \eqref{fs} we have
$$\lambda=1+\frac{|f_{S^{n-1}(r)}(x)-f_{S^{n-1}(r)}(y)|}{|f_{S^{n-1}(r)}(y)|}=1+\frac{|x-y|}{|x|}=\frac{|x|+|x-y|}{|x|}.$$
The inequality \eqref{qfm2} is clear.
\end{proof}

By taking $a=1$ in Theorem \ref{star1}, Lemma \ref{deqin}, Proposition \ref{pr} and Lemma \ref{delta}--\ref{lees}, we have the following results.
\begin{theorem}\label{qinv1}
Let $\mathcal{M}$ be the boundary of a domain in $\RN$, strictly starlike w.r.t. the origin which satisfies the $\alpha$-tangent condition. Let $f_\mathcal{M}: \overline{\R^n}\rightarrow\overline{\R^n}$  be the quasi-inversion in $\mathcal{M}$. Then $f_\mathcal{M}$ is an $L^6$-bi-Lipschitz map w.r.t. the chordal metric, where $L$ is of the form as in \eqref{qinl} of Theorem \ref{star1}. Moreover, if $\mathcal{M}_t$, $t\in [0,1]$, is a family of surfaces smoothly converging to the unit sphere $S^{n-1}$, then the bi-Lipschitz constants $L_t$ of the quasi-inversions $f_{\mathcal{M}_t}$ tend to $1$.
\end{theorem}

\begin{theorem}\label{qinv2}
Let $\mathcal{M}$ be the boundary of a domain in $\RN$, strictly starlike w.r.t. the origin which satisfies the $\alpha$-tangent condition. Let $G\subsetneq\R^n$ be a domain and $f_\mathcal{M}: G\rightarrow f_\mathcal{M} G$ be the quasi-inversion in $\mathcal{M}$. Then $f_\mathcal{M}$ is an $L^8$-bi-Lipschitz map w.r.t. both the M\"obius metric and Ferrand's metric, where $L$ is of the form as in \eqref{qinl} of Theorem \ref{star1}. Moreover, if $\mathcal{M}_t$, $t\in [0,1]$, is a family of surfaces smoothly converging to the unit sphere $S^{n-1}$, then the bi-Lipschitz constants $L_t$ of the quasi-inversions $f_{\mathcal{M}_t}$ tend to $1$.
\end{theorem}

By taking $a=\csc\alpha$ in Theorem~\ref{star1}, Lemma~\ref{deqin} and the fact that the inversion w.r.t. the unit sphere is a conformal mapping we obtain
\begin{theorem}\label{qinv3}
Let $\mathcal{M}$ be the boundary of a domain in $\RN$, strictly starlike w.r.t. the origin which satisfies the $\alpha$-tangent condition. Let $f_\mathcal{M}: \overline{\R^n}\rightarrow\overline{\R^n}$  be the quasi-inversion in $\mathcal{M}$. Then $f_\mathcal{M}$ is a quasiconformal mapping with
$$K=\cot^2\frac{\alpha}{2}.$$
If $n=2$, then $k=(K-1)/(K+1)=\cos \alpha$ and this bound coincides with that of Fait, J. G. Krzy\.z, and J. Zygmunt in \cite{fkz}. Furthermore, if $\mathcal{M}_t$, $t\in [0,1]$, is a family of surfaces smoothly converging to the sphere $S^{n-1}(r)$, $r>0$, then the quasiconformality constants $K_t$ of the quasi-inversions $f_{\mathcal{M}_t}$ tend to $1$.
\end{theorem}

\begin{remark}
Professor T. Sugawa pointed out to the authors  a connection between
the $\alpha$-tangent condition and strongly starlike plane domain
of order $1-(2/\pi) \alpha$ in the sense of Brannan-Kirwan and Stankiewicz. Let us sketch the proof of this connection.  We say that (see [Su]) a complex domain $D$ is strongly starlike of order $\delta\in[0,1]$ with respect to $0$ if the conformal mapping $f$, of the unit disk onto $D$, with $f(0)=0$ satisfies the condition $$\mathrm{arg}\frac{zf'(z)}{f(z)}\le \delta\frac{\pi}{2}.$$ Assume for the simplicity that $\partial D$ is smooth. Then $f'$ is continuous up to the boundary.
Let $f(e^{it})=\rho(t)e^{i\varphi(t)}$. Then we obtain $$f'(e^{it})=-i (\rho'(t)+i \rho(t)\varphi'(t))e^{i\varphi(t)-it}.$$ Thus for $z=e^{it}$, $$\frac{zf'(z)}{f(z)}=\varphi'(t)+i\frac{\rho'(t)}{\rho(t)}.$$ Let $r(s)=\rho(\varphi^{-1}(s))$. Then for $z=e^{it}$,  if $r'(s)\ge 0$, then $\lambda_s=\alpha(f(z))$ and by \eqref{rdot}, we have $$\mathrm{arg}\frac{zf'(z)}{f(z)}=\mathrm{arg}(r(s)+i r'(s))= \arctan\frac{r'(s)}{r(s)}=\pi/2-\mathrm{arccot} \frac{r'(s)}{r(s)}=\pi/2-\lambda_s.$$ So
$$\mathrm{arg}\frac{zf'(z)}{f(z)}\le \pi/2-\alpha=(1-2\alpha/\pi)\frac{\pi}{2}.$$
So  $\alpha-$tangent condition implies $1-2\alpha/\pi$ strongly starlike condition in the plane and viceversa.
\end{remark}

\bigskip
{\bf Acknowledgement.} We are grateful to the anonymous referee for very constructive comments that have improved this paper. This paper was written during the visit of the first author to the University of Turku in the framework of programme JoinEU-SEE III, ERASMUS MUNDUS. The authors are indebted to Prof. T. Sugawa
for bringing the papers \cite{fkz} and \cite{su} to our attention. The research of the second and the third authors was supported by the Academy of Finland,
Project 2600066611.


\begin{thebibliography}{BVKV}
\bibitem[Ah]{ah2}
{\sc  L. V. Ahlfors:}
{\it M\"obius transformations in several dimensions.}
 Ordway Professorship Lectures in Mathematics. University of Minnesota, School of Mathematics, Minneapolis, Minn., 1981.


\bibitem[An]{ander}
{\sc G. D. Anderson:}
{\it The coefficients of quasiconformality of ellipsoids.}
Ann. Acad. Sci. Fenn. Ser. A I No.  411 (1967), 14 pp.

\bibitem[B]{be}
{\sc A. F. Beardon:}
{\it The geometry of discrete groups.}
Graduate Texts in Mathematics, 91. Springer-Verlag, New York, 1983.
\bibitem[BBKV]{bvkv}
{\sc B. A. Bhayo, V. Bo\v{z}in. D. Kalaj and M. Vuorinen:}
{\it Norm inequalities for vector functions.}
J. Math. Anal. Appl. 380 (2011), 768-781.
\bibitem[C]{car}
{\sc P. Caraman:}
{\it $n$-dimensional quasiconformal (QCf) mappings.}
Revised, enlarged and translated from the Romanian by the author.
Editura Academiei Rom\^ane, Bucharest; Abacus Press, Tunbridge Wells; Haessner Publishing, Inc., Newfoundland, N. J., 1974. 

\bibitem[FKZ]{fkz}
{\sc M. Fait, J. G. Krzy\.z, and J. Zygmunt:}
{\it Explicit quasiconformal extensions for some classes of
univalent functions.}
Comment. Math. Helv.  51 (1976), no. 2, 279--285.

\bibitem[GV]{gv}
{\sc F. W. Gehring and J. V\"ais\"al\"a:}
{\it The coefficients of quasiconformality of domains in space. }
 Acta Math. 114 (1965), 1--70.


\bibitem[K1]{comp}
{\sc D. Kalaj:}
{\it On harmonic diffeomorphisms of the unit disc onto a convex domain.}
Complex Var. Theory Appl.  48 (2003), no. 2, 175--187.

\bibitem[K2]{comp1}
{\sc D. Kalaj:}
{\it Radial extension of a bi-Lipschitz parametrization of a starlike Jordan curve.}
Complex Var. Theory Appl. 59 (2014), no. 6, 809–-825.

\bibitem[L]{misha}
{\sc G. Leoni:}
{\it A first course in Sobolev spaces.}
Graduate Studies in Mathematics, 105. American Mathematical Society, Providence, RI, 2009.

\bibitem[LV]{lv}
{\sc J. Luukkainen and J. V\"ais\"al\"a:}
{\it Elements of Lipschitz topology.}
Ann. Acad. Sci. Fenn. Ser. A I Math.  3 (1977), no. 1, 85--122.

\bibitem[MS]{mar2}
{\sc O. Martio and U. Srebro:}
{\it On the existence of automorphic quasimeromorphic mappings in $\R^n$.}
 Ann. Acad. Sci. Fenn. Ser. A I Math.  3 (1977), no. 1, 123--130.


\bibitem[S1]{s1}
{\sc P. Seittenranta:}
{\it M\"obius invariant metrics and quasiconformal maps.}
Lic. thesis Univ. of Helsinki, 1996(unpublished).


\bibitem[S2]{s2}
{\sc P. Seittenranta:}
{\it M\"obius-invariant metrics. }
 Math. Proc. Cambridge Philos. Soc.  125 (1999), no. 3, 511--533.
\bibitem[Su]{su}
{\sc T.  Sugawa:}
{\it A self-duality of strong starlikeness.}
Kodai Math. J. 28 (2005), no. 2, 382--389.

\bibitem[Va]{vais}
{\sc J. V\"ais\"al\"a:}
{\it Lectures on n-dimensional quasiconformal mappings.}
Lecture Notes in Mathematics, Vol. 229. Springer-Verlag, Berlin-New York, 1971.

\bibitem[Vu]{vu}
{\sc M. Vuorinen:}
{\it Conformal geometry and quasiregular mappings.}
Lecture Notes in Mathematics, 1319. Springer-Verlag, Berlin, 1988.
\end{thebibliography}
\end{document}